\documentclass[final,11pt,a4paper]{amsart} 

\usepackage{tikz,tikz-cd}
\usetikzlibrary{snakes}
\usetikzlibrary{shapes.geometric,positioning}
\usetikzlibrary{arrows,decorations.pathmorphing,decorations.pathreplacing}

\definecolor{teal}{rgb}{0.0, 0.5, 0.5}
\definecolor{tealblue}{rgb}{0.21, 0.46, 0.53}
\definecolor{tealgreen}{rgb}{0.0, 0.51, 0.5}
\definecolor{tuscanred}{rgb}{0.51, 0.21, 0.21}
\definecolor{sangria}{rgb}{0.57, 0.0, 0.04}
\definecolor{rufous}{rgb}{0.66, 0.11, 0.03}
\definecolor{pinegreen}{rgb}{0.0, 0.47, 0.44}
\definecolor{darkscarlet}{rgb}{0.34, 0.01, 0.1}
\definecolor{darkseagreen}{rgb}{0.56, 0.74, 0.56}
\definecolor{darkpastelred}{rgb}{0.76, 0.23, 0.13}
\definecolor{darkpink}{rgb}{0.91, 0.33, 0.5}
\definecolor{darkpastelblue}{rgb}{0.47, 0.62, 0.8}
\definecolor{alizarin}{rgb}{0.82, 0.1, 0.26}
\definecolor{candyapplered}{rgb}{1.0, 0.03, 0.0}

\usepackage{graphicx}
\usepackage[all]{xy}
\usepackage{placeins}
\usepackage{enumitem}
\usepackage{amssymb}
\usepackage{latexsym}
\usepackage{amsmath}
\usepackage{mathrsfs}
\usepackage{array,booktabs}
\usepackage{verbatim}
\usepackage{fullpage}
\usepackage[notref, notcite]{showkeys}  
\usepackage{color}
\usepackage{float}
\usepackage{comment}
\usepackage[normalem]{ulem}

\usepackage{hyperref}
\hypersetup{
    colorlinks,
    linkcolor={red!50!black},
    citecolor={blue!50!black},
    urlcolor={blue!80!black}
}
\newcommand{\harxiv}[1]{ \href{http://arxiv.org/abs/#1}{\texttt{arXiv:#1}}}
\newcommand{\hyref}[2]{ \hyperref[#2]{#1~\ref*{#2}} }

\newcommand{\Canakci}{\c{C}anak\c{c}\i}
\renewcommand{\comment}[1]{{}}

\theoremstyle{plain}
\newtheorem{theorem}{Theorem}[section]

\newtheorem{lemma}[theorem]{Lemma}
\newtheorem{corollary}[theorem]{Corollary}
\newtheorem{proposition}[theorem]{Proposition}
\newtheorem*{introtheorem}{Theorem}

\theoremstyle{definition}
\newtheorem{remark}[theorem]{Remark}
\newtheorem{example}[theorem]{Example}
\newtheorem{definition}[theorem]{Definition}
\newtheorem{definitions}[theorem]{Definitions}

\newtheorem*{convention}{Convention}

\newtheorem*{introremark}{Remark}

\newcommand{\sD}{\mathsf{D}}
\newcommand{\sK}{\mathsf{K}}

\DeclareMathAlphabet{\mathpzc}{OT1}{pzc}{m}{it}

\newcommand{\Db}{\sD^b}
\newcommand{\KminusL}{\sK^{b,-}(\proj{\Lambda})}
\newcommand{\KbL}{\sK^b(\proj{\Lambda})}

\newcommand{\kk}{{\mathbf{k}}}

\renewcommand{\emptyset}{\varnothing}

\DeclareMathOperator{\Hom}{\mathrm{Hom}}
\DeclareMathOperator{\Ext}{\mathrm{Ext}}

\newcommand{\proj}[1]{\mathsf{proj}(#1)}

\newcommand{\length}[1]{\mathrm{len}(#1)}

\newcommand{\B}{B^{\bullet}}
\renewcommand{\P}{P^{\bullet}}
\newcommand{\Q}{Q^{\bullet}}
\newcommand{\M}{M^{\bullet}}
\newcommand{\E}{E^{\bullet}}

\newcommand{\f}{f^{\bullet}}
\newcommand{\g}{g^{\bullet}}

\renewcommand{\i}{i^{\bullet}}

\newcommand{\id}{\mathsf{id}^{\bullet}}

\newcommand{\xydot}{{\bullet}}
\newcommand{\arr}{\ar@{-}[r]}

\newenvironment{pmat}{\left[ \begin{smallmatrix}}{\end{smallmatrix} \right]}

\renewcommand{\phi}{\varphi}
\renewcommand{\epsilon}{\varepsilon}

\begin{document}

\title[Mapping cones]{Mapping cones in the bounded derived category of a gentle algebra}

\author{\.{I}lke \Canakci} 
\address{School of Mathematics, Statistics and Physics, Newcastle University, Newcastle Upon Tyne, NE1 7RU, United Kingdom.}
\email{ilke.canakci@ncl.ac.uk}

\author{David Pauksztello} 
\address{Department of Mathematics and Statistics, Lancaster University, Lancaster, LA1 4YF, United Kingdom}
\email{d.pauksztello@lancaster.ac.uk}

\author{Sibylle Schroll}
\address{Department of Mathematics, University of Leicester, University Road, Leicester, LE1 7RH, United Kingdom.}
\email{schroll@leicester.ac.uk}

\keywords{bounded derived category, gentle algebra, homotopy string and band, string combinatorics, mapping cone
\newline This work was supported through EPSRC grants EP/K026364/1, EP/K022490/1, EP/N005457/1 and EP/P016294/1. 
}


\begin{abstract}
In this article we  describe the triangulated structure of the bounded derived category of a gentle algebra by describing the triangles induced by the morphisms between indecomposable objects in a basis of their Hom-space.
\end{abstract}

\maketitle

\vspace{-0.5cm}

\section*{Introduction}

Derived categories provide a common framework for homological algebra in subjects such as  algebra, geometry and mathematical physics. For example, in mathematical physics, in the context of homological mirror symmetry,  they are  the natural setting for Bridgeland's stability conditions \cite{Bridgeland}. In algebraic geometry, they arise in the study of non-commutative crepant resolutions which are often studied via an algebra whose derived category is  equivalent  to the derived category of the smooth variety resolving the singularity \cite{vdB}.  
In representation theory, derived categories are the natural setting for tilting theory of finite dimensional algebras, see for example \cite{Happel}.  
Thus understanding the structure of derived categories and their properties is an important problem.
However, owing to their complexity, in general, this is difficult to achieve. Therefore in the cases, where this is achievable,  it is of great value to obtain as much detailed  knowledge  of the derived category as possible.

In the context of  the representation theory of finite dimensional algebras, we will now describe a situation where it is possible to gain insight into the structure of derived categories. According to the tame-wild dichotomy \cite{Drozd},  algebras are either  of tame representation type,  that is, all indecomposable finite-dimensional modules may be classified, or they are of wild representation type and it is considered that a complete classification is impossible. 
 Therefore much of the work in representation theory has been focused on tame algebras and for particular classes of tame algebras we have a good understanding of their (classical) representation theory, that is of their modules categories. 
  A good example of this are special biserial algebras \cite{WW}, a prominent class of tame algebras, that has been widely studied,  with many exciting recent developments, see for example \cite{ AG, AmLP, AKE, Carroll, Cibils, CSP,  Erdmann, GM,  Schroll, Xu}.  One of the reasons that special biserial algebras are so well-understood is that their indecomposable representations are classified in terms of strings and bands \cite{GP, Ringel, WW}. 
 
The notion of derived-tameness was introduced in \cite{GK}. In the case of derived-tame algebras we can gain concrete insight into the structure of the derived category.
Gentle algebras are derived tame \cite{BM}. They form a subclass of special biserial algebras which is closed under derived equivalence \cite{Schroer-Zimmermann}. They are, therefore, a natural class of algebras whose derived category is the object of intensive study. As a result, our understanding of the structure of derived categories of gentle algebras parallels that of our understanding of their module categories. 

Gentle algebras first arose in the setting of tilting theory in the classification of iterated tilted algebras of type $A$ and type $\widetilde{A}$ in \cite{Assem} and \cite{AS} respectively.  They now play an important role in many areas of mathematics: in algebra, they occur in cluster theory as Jacobian algebras associated to surface triangulations \cite{ABCP, FST, Labardini},  and in recent advances in invariant theory \cite{Carroll-Chindris}. 
In addition to their widespread appearance in representation theory and algebra, gentle algebras also are increasingly ubiquituous in geometric contexts. For example, their singularity category \cite{Buchweitz} -- which measures how far an algebra or variety is away from being nonsingular -- was described in \cite{Kalck}. They feature prominently in the programme to understand singularities of nodal curves \cite{BD,BD2,BD3} and  in an algebraic approach to mirror symmetry \cite{Bocklandt, HKK, Lekili-Polishchuk, Opper-Plamondon-Schroll}.

From now on let $\Lambda$ be a gentle algebra and $\Db(\Lambda)$ be its bounded derived category with shift functor $\Sigma$. The indecomposable objects in $\Db(\Lambda)$ have been classified \cite{BM} in terms of string combinatorics: namely they are given in terms of \emph{homotopy strings} and \emph{homotopy bands}, the terminology originating in \cite{Bo}; see also \cite{BD} for a similar approach in the context of nodal algebras.  
The corresponding indecomposable complexes are called \emph{string complexes} and \emph{band complexes}, respectively.

Using the Happel functor \cite{Happel}, the Auslander--Reiten (AR) structure of the perfect  category $\KbL$ was determined in \cite{Bo}.
See also \cite[\S 6]{ALP} for similar results without the use of the Happel functor. A canonical basis for the morphisms between string complexes and (one-dimensional) band complexes in $\Db(\Lambda)$ was given  in \cite{ALP} in terms of three types of morphism: \emph{graph maps}, \emph{single maps} and \emph{double maps}. 
We will call this the \emph{standard basis}.

A general description of the middle terms of triangles of the form $\Q \to \E \to \P \to \Sigma \Q$, where $\P$ and $\Q$ are indecomposable objects, is not known. In this article, using (homotopy) string combinatorics, we describe the mapping cones of elements of the standard basis for the morphisms between string complexes and (one-dimensional) band complexes. As such, we explicitly describe the middle terms of extensions for a standard basis of $\Ext^1_{\Db(\Lambda)}(\P,\Q) := Hom^1_{\Db(\Lambda)}(\P, \Sigma\Q)$ where $\P$ and $\Q$ are string or band complexes. Understanding middle terms of extensions is important for many applications, 
for example:  in cluster theory, these form the basis of the so-called exchange triangles and exchange relations \cite{BMRRT,CCS,CS,IT,IY}; 
understanding middle terms of extensions is a key ingredient in classifying torsion pairs in triangulated categories \cite{Coelho-P,GHJ,Ng,Qiu-Zhou,Zhang-Zhou-Zhu}; 
the description of middle terms of extensions in the derived category of derived-discrete algebras is instrumental in the classification of thick subcategories of discrete derived categories \cite{Broomhead} and  they have been used in \cite{CPS} to show that the short exact sequences given in \cite{Schroer} form a  basis of the $\Ext^1$-space between any two indecomposable $\Lambda$-modules, thus answering this longstanding open question.
Moreover, the results of this paper have been applied in \cite{Opper-Plamondon-Schroll} to give a geometric interpretation of mapping cones in terms of resolutions of crossing curves in the context of a geometric model of the derived category of a gentle algebra.

In $\Db(\Lambda)$ the middle term $\E$ of an extension  $\Q \to \E \to \P \stackrel{\f}{\to} \Sigma \Q$ is given by the inverse shift of the mapping cone $\M_{\f}$ of the map $\f$.
In this paper, we describe the indecomposable summands of the mapping cones  of the standard basis elements, that is of maps $\f$ in the canonical basis of $\Hom_{\Db(\Lambda)}(\P,\Q)$, where $\P$ and $\Q$ are  indecomposable complexes. Our description is  general in that our results cover both string and band complexes. Our main results can be summarised in the following 
graphical description  of the mapping cones of  standard basis elements:
 
\begin{introtheorem}
Let $\Lambda$ be a gentle algebra. Suppose $\Q_\sigma$ and $\Q_\tau$ are indecomposable objects in $\Db(\Lambda)$, where  $\sigma$ and $\tau$ are the corresponding homotopy strings or bands. Let $\f \in \Hom_{\Db(\Lambda)}(\Q_\sigma, \Q_\tau)$ be a standard basis element. Then the indecomposable summands of the mapping cone $\M_{\f}$ are given by the homotopy strings and bands occurring in the red (dotted) and green (dashed) boxes resulting from
the following graphical calculus.
 
 \begin{enumerate}
 \item Let $\sigma = \beta \sigma_L \rho \sigma_R \alpha$ and $\tau = \delta \tau_L \rho \tau_R \gamma$ and suppose $\f$ is a graph map corresponding to the overlap $\rho$. Then $\M_{\f}$ is given by:
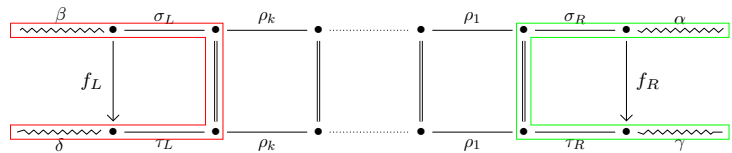
\[
\scalebox{.9}{\begin{tikzpicture}[scale=1.5]
\node (A0) at (-1,0) {};
\node[scale=.7] (A1) at (0,0) {$\bullet $};
\node[scale=.7] (A2) at (1,0) {$\bullet $};
\node[scale=.7] (A3) at (2,0) {$\bullet $};
\node[scale=.7] (A4) at (3,0) {$\bullet $};
\node[scale=.7] (A5) at (4,0) {$\bullet $};
\node[scale=.7] (A6) at (5,0) {$\bullet $};
\node[scale=.7] (A7) at (6,0) {};
\node[scale=.7] (B0) at (-1,-1) {};
\node[scale=.7] (B1) at (0,-1) {$\bullet $};
\node[scale=.7] (B2) at (1,-1) {$\bullet $};
\node[scale=.7] (B3) at (2,-1) {$\bullet $};
\node[scale=.7] (B4) at (3,-1) {$\bullet $};
\node[scale=.7] (B5) at (4,-1) {$\bullet $};
\node[scale=.7] (B6) at (5,-1) {$\bullet $};
\node[scale=.7] (B7) at (6,-1) {};
\path[color=white] (A0) edge node[above,color=black,scale=.7]{$\beta $}(A1)
 (B0) edge node[below,color=black,scale=.7]{$\delta $}(B1)
 (A6) edge node[above,color=black,scale=.7]{$\alpha $}(A7)
 (B6) edge node[below,color=black,scale=.7]{$\gamma $}(B7);
\draw [line join=round,
decorate, decoration={
    zigzag,
    segment length=4,
    amplitude=.9,post=lineto,
    post length=2pt
}] (A0) --  (A1)
(B0)--(B1)
(A6)--(A7)
(B6)--(B7);
\path
(A1) edge node[above,scale=.7]{$\sigma _{L}$} (A2)
(A2) edge node[above,scale=.7]{$\rho _{k}$} (A3)
(A4) edge node[above,scale=.7]{$\rho _{1}$} (A5)
(A5) edge node[above,scale=.7]{$\sigma _{R}$} (A6);
\path[->,font=\scriptsize ,>=angle 90]
(A1) edge node[left]{$f_{L}$} (B1)
(A6) edge node[right]{$f_{R}$} (B6);
\path[font=\scriptsize ,>=angle 90]
(A2)  edge[double]  (B2)
(A3)  edge[double]  (B3)
(A4)  edge[double]  (B4)
(A5)  edge[double]  (B5);
\draw [densely dotted] (A3)--(A4) (B3)--(B4);
\path
(B1) edge node[below,scale=.7]{$\tau _{L}$} (B2)
(B2) edge node[below,scale=.7]{$\rho _{k}$} (B3)
(B4) edge node[below,scale=.7]{$\rho _{1}$} (B5)
(B5) edge node[below,scale=.7]{$\tau _{R}$} (B6);
\draw [color=red, dotted, line width=1.2] (-1,.07)--(1.07,.07)--(1.07,-1.07)--(-1,-1.07)--(-1, -.93)--(.9,-.93)--(.9,-.07)--(-1,-.07)--(-1,.07);
\draw [color=green, dashed, line width=1.1] (3.93,.07)--(6,.07)--(6,-.07)--(4.07,-.07)--(4.07, -.93)--(6,-.93)--(6,-1.07)--(3.93,-1.07)--(3.93,.07);
\end{tikzpicture}}
\]
\item Let $\sigma= \beta \sigma_L \sigma_R \alpha$ and $\tau = \delta \tau_L \tau_R \gamma$ and suppose $\f$ is a single map. Then $\M_{\f}$ is given by: 
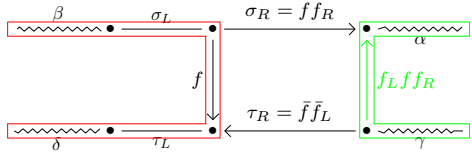
\[
\scalebox{.9}{\begin{tikzpicture}[scale=1.5]
\node (A0) at (-1,0) {};
\node[scale=.7] (A1) at (0,0) {$\bullet $};
\node[scale=.7] (A2) at (1,0) {$\bullet $};
\node[scale=.7] (A3) at (2.5,0) {$\bullet $};
\node[scale=.7] (A4) at (3.5,0) {};
\node[scale=.7] (B0) at (-1,-1) {};
\node[scale=.7] (B1) at (0,-1) {$\bullet $};
\node[scale=.7] (B2) at (1,-1) {$\bullet $};
\node[scale=.7] (B3) at (2.5,-1) {$\bullet $};
\node[scale=.7] (B4) at (3.5,-1) {};
\path[color=white]
 (A0) edge node[above,color=black,scale=.7]{$\beta $}(A1)
 (A3) edge node[below,color=black,scale=.7]{$\alpha $}(A4)
 (B0) edge node[below,color=black,scale=.7]{$\delta $}(B1)
 (B3) edge node[below,color=black,scale=.7]{$\gamma $}(B4);
\draw [line join=round,
decorate, decoration={
    zigzag,
    segment length=4,
    amplitude=.9,post=lineto,
    post length=2pt
}] (A0) --  (A1)
(A3) -- (A4)
(B0)--(B1)
(B3)--(B4);
\path (A1) edge node[above,scale=.7]{$\sigma _{L}$} (A2);
\path[->,font=\scriptsize ,>=angle 90]
(A2) edge node[left]{$f$} (B2)
(A2) edge node[above]{$\sigma _{R}=ff_{R}$} (A3)
(B3) edge node[above]{$\tau _{R}=\bar f \bar f_{L}$} (B2);
\path (B1) edge node[below,scale=.7]{$\tau _{L}$} (B2);
\draw [color=red, dotted, line width=1.2] (-1,.07)--(1.07,.07)--(1.07,-1.07)--(-1,-1.07)--(-1, -.93)--(.93,-.93)--(.93,-.07)--(-1,-.07)--(-1,.07);
\draw [color=green, dashed, line width=1.1] (2.57,-.6)--(2.57,-.07)--(3.5,-.07)--(3.5,.07)--(2.43,.07)--(2.43,-1.07) (2.57,-.6)--(2.57,-.93)--(3.5,-.93)--(3.5,-1.07)--(2.43,-1.07);
\path[->,font=\scriptsize ,>=angle 90,color=green,scale=.7]
(B3) edge node[right]{$f_{L}ff_{R}$} (A3);
\end{tikzpicture}}
\]

\item Let $\sigma = \beta \sigma_L \sigma_C \sigma_R \alpha$ and $\tau = \delta \tau_L \tau_C \tau_R \gamma$ and suppose $\f$ is a double map. Then $\M_{\f}$ is given by:

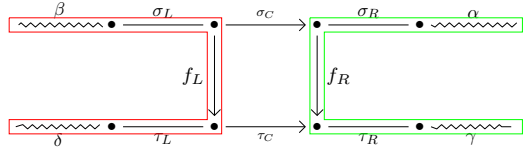
\[
\scalebox{.9}{\begin{tikzpicture}[scale=1.5]
\node (A0) at (-1,0) {};
\node[scale=.7] (A1) at (0,0) {$\bullet $};
\node[scale=.7] (A2) at (1,0) {$\bullet $};
\node[scale=.7] (A3) at (2,0) {$\bullet $};
\node[scale=.7] (A4) at (3,0) {$\bullet $};
\node[scale=.7] (A5) at (4,0) {};
\node[scale=.7] (B0) at (-1,-1) {};
\node[scale=.7] (B1) at (0,-1) {$\bullet $};
\node[scale=.7] (B2) at (1,-1) {$\bullet $};
\node[scale=.7] (B3) at (2,-1) {$\bullet $};
\node[scale=.7] (B4) at (3,-1) {$\bullet $};
\node[scale=.7] (A5) at (4,0) {};
\path[color=white] (A0) edge node[above,color=black,scale=.7]{$\beta $}(A1)
 (B0) edge node[below,color=black,scale=.7]{$\delta $}(B1)
 (A4) edge node[above,color=black,scale=.7]{$\alpha $}(A5)
 (B4) edge node[below,color=black,scale=.7]{$\gamma $}(B5);
\draw [line join=round,
decorate, decoration={
    zigzag,
    segment length=4,
    amplitude=.9,post=lineto,
    post length=2pt
}] (A0) --  (A1)
(B0)--(B1)
(A4)--(A5)
(B4)--(B5);
\path (A1) edge node[above,scale=.7]{$\sigma _{L}$} (A2);
\draw (A3) edge node[above,scale=.7]{$\sigma _{R}$} (A4);
\draw (B3) edge node[below,scale=.7]{$\tau _{R}$} (B4);
\path[->,font=\scriptsize ,>=angle 90]
(A2) edge node[left]{$f_{L}$} (B2)
(A3) edge node[right]{$f_{R}$} (B3)
(A2) edge node[above,scale=.7]{$\sigma _{C}$} (A3)
(B2) edge node[below,scale=.7]{$\tau _{C}$} (B3);
\path (B1) edge node[below,scale=.7]{$\tau _{L}$} (B2);
\draw [color=red, dotted, line width=1.2] (-1,.07)--(1.07,.07)--(1.07,-1.07)--(-1,-1.07)--(-1, -.93)--(.93,-.93)--(.93,-.07)--(-1,-.07)--(-1,.07);
\draw [color=green, dashed, line width=1.1] (1.93,.07)--(4,.07)--(4,-.07)--(2.07,-.07)--(2.07, -.93)--(4,-.93)--(4,-1.07)--(1.93,-1.07)--(1.93,.07);
\end{tikzpicture}}
\]

\end{enumerate}
\end{introtheorem}
 
We refer the reader to Sections~\ref{sec:graph-maps}, \ref{sec:single-maps} and \ref{sec:double-maps} for precise statements and details.

\begin{convention}
Throughout this article, all modules will be  left modules and all maps will be composed from left to right.
\end{convention}

\begin{introremark}
The statement of the theorem above contains two oversights in the case that a band complex is involved: for a (quasi-)graph map the overlap $\rho$ may be longer than the band, and when the result of the mapping cone calculus is a cyclic string, it may be a power of a band. These oversights are corrected in \cite{Addendum}. We have added footnotes to the affected statements indicating where one can find the corrected statements.
\end{introremark}

\section{Background}

\subsection{The homotopy category and mapping cones}

The required background on derived and triangulated categories in the setting of representation theory can be found in \cite{Happel}.

Let $\Lambda$ be a finite-dimensional algebra over an algebraically closed field $\kk$. The category of interest in this article will be $\KminusL$, the homotopy category of right bounded complexes of projective left $\Lambda$-modules with bounded cohomology. The morphisms of $\KminusL$ are cochain maps of complexes up to homotopy. 
It is well known that $\KminusL \simeq \Db(\Lambda)$, the bounded derived category of finitely generated left $\Lambda$-modules and for convenience we shall always identify $\Db(\Lambda)$ with $ \KminusL$. 

\begin{definition}
Let $(\P, d_P)$ and $(\Q,d_Q)$ be complexes in $\KminusL$ and suppose $\f \colon \P \to \Q$ is a cochain map. The \emph{mapping cone of $\f$} is the complex $(\M_{\f},d_{M_{\f}})$ given by
\[
M_f^n = P^{n+1} \oplus Q^n
\quad \text{and} \quad
d^n_{M_{\f}} =
\left[
\begin{array}{cc}
-d^{n+1}_P & f^{n+1} \\
0                 & d^n_Q
\end{array}
\right]
\]
\end{definition}

\subsection{Gentle algebras}

The definition of a gentle algebra goes back to \cite{AS}, where they first occurred as iterated tilted algebras of type $A$.

\begin{definition}
A finite-dimensional $\kk$-algebra is called \emph{gentle} if it is Morita equivalent to an algebra $\kk Q/I$ where $Q$ is a quiver and $I$ an admissible ideal in $\kk Q$ such that 
\begin{enumerate}
\item For each vertex $i \in Q_0$ there are at most two arrows starting at $i$ and at most two arrows ending at $i$;
\item For each arrow $a \in Q_1$ there is at most one arrow $b$ such that $ba \notin I$ and at most one arrow $c$ such that $ac \notin I$;
\item For each arrow $a \in Q_1$ there is at most one arrow $b$ with $s(b) = e(a)$ and $ba \in I$ and at most one arrow $c$ with $e(c) = s(a)$ and $ac \in I$;
\item The ideal $I$ is generated by length two monomial relations.
\end{enumerate} 
\end{definition}

Note that if one removes condition $(3)$ and relaxes condition $(4)$ so that $I$ is generated simply by monomial relations one obtains a so-called \emph{string algebra}. 

From now on,  $\Lambda$ will be a gentle algebra.

\subsection{String and band complexes}

The indecomposable complexes in $\Db(\Lambda)$ are parametrised by homotopy strings and bands. The original reference is \cite{BM} but we use the notation and terminology of \cite{ALP} dropping  some of the formality regarding cohomological degrees.
All modules will be finitely generated left modules and therefore paths in the quiver will be read from right to left.
For each $a \in Q_1$ we define a formal inverse arrow $\overline{a} = a^{-1}$ with $s(\overline{a}) = e(a)$ and $e(\overline{a}) = s(a)$. For a path $p = a_n \cdots a_1$ in $(Q,I)$ the inverse path is $\overline{p} = \overline{a_1} \cdots \overline{a_n}$.

\begin{definitions}
We recall the following definitions.
\begin{enumerate}
\item A \emph{walk} in $(Q,I)$ is a sequence $w = w_l \cdots w_1$ satisfying $s(w_{i+1}) = e(w_i)$, where each $w_i$ is either an arrow or an inverse arrow, and where the sequence does not contain any subsequence of the form $a \overline{a}$ or $\overline{a} a$ for some $a \in Q_1$. 
\item A \emph{(finite) homotopy string} is a walk of finite length in $(Q,I)$. In addition there are \emph{trivial homotopy strings} for each vertex $x \in Q_0$. 
\item A subwalk $p = w_j \cdots w_i$ of a homotopy string $\sigma = w_l \cdots w_1$ is a \emph{homotopy letter} if 
\begin{enumerate}
\item $p$ or $\overline{p}$ is a path in $(Q,I)$; and,
\item $w_i, \overline{w_{i-1}} \in Q_1$, or $\overline{w_i}, w_{i-1} \in Q_1$, or $w_i w_{i-1} \in I$, or $\overline{w_{i-1}} \overline{w_i} \in I$; and,
\item $w_j, \overline{w_{j+1}} \in Q_1$, or $\overline{w_j}, w_{j+1} \in Q_1$, or $w_{j+1} w_j \in I$, or $\overline{w_j} \overline{w_{j+1}} \in I$.
\end{enumerate}
The subwalk $p$ is a \emph{direct homotopy letter} if it is a path in $(Q,I)$ and an \emph{inverse homotopy letter} if $\overline{p}$ is a path in $(Q,I)$.
This partitions a homotopy string $\sigma$ into homotopy letters and we write $\sigma = \sigma_n \cdots \sigma_1$ for this decomposition. 
A \emph{homotopy subletter} of $p$ is a subwalk of $p$. 
\item A \emph{trivial homotopy letter} is one of the form $p = 1_x$ for some $x \in Q_0$. Sometimes we shall assert the nonexistence of homotopy letters, and in this case we write $p = \emptyset$.
\item Let $\sigma = \sigma_n \cdots \sigma_1$ be a homotopy string decomposed into its homotopy letters. A subwalk $\tau = \sigma_j \cdots \sigma_i$ with $1 \leq i \leq j \leq n$ is called a \emph{homotopy substring} of $\sigma$.
\item A \emph{homotopy band} is a homotopy string $\sigma = \sigma_n \cdots \sigma_1$ with $s(\sigma) = e(\sigma)$, $\sigma_1 \neq \overline{\sigma_n}$, $\sigma \neq \tau^m$ for some homotopy substring $\tau$ and $m >1$, and $\sigma$ has equal numbers of direct and inverse homotopy letters.
\item A (possibly infinite) walk $w = w_n \cdots w_1$ is called a \emph{direct antipath} if each $w_i \in Q_1$ is a direct homotopy letter; it is called an \emph{inverse antipath} if each $\overline{w_i} \in Q_1$ and each $w_i$ is an inverse homotopy letter.
\item A left infinite walk $w = \cdots w_n \cdots w_2 w_1$ is a \emph{left infinite homotopy string} if there exists $m \geq 1$ such that $v = \cdots w_n \cdots w_{m+1} w_m$ is a direct antipath.
\item A right infinite walk $w = w_{-1} w_{-2} \cdots w_{-n} \cdots$ is a \emph{right infinite homotopy string} if there exists $m \geq 1$ such that $v = w_{-m} w_{-m-1} \cdots w_{-n} \cdots$ is an inverse antipath.
\item A two-sided infinite walk $w = \cdots w_2 w_1 w_0 w_{-1} \cdots$ is called a \emph{two-sided infinite homotopy string} if there exist integers $n > m$ such that $\cdots v_{n+1} v_n$ is a direct antipath and $v_m v_{m-1} \cdots$ is an inverse antipath.
\item A two-sided, left or right infinite homotopy string is called an \emph{infinite homotopy string}.
\end{enumerate}  
\end{definitions}

Let $\sigma$ be a (possibly infinite) homotopy string. Modulo the equivalence relation $\sigma \sim \overline{\sigma}$, the set of homotopy strings determine the \emph{string complexes} $\P_\sigma$. 
Up to the equivalence relation given by inversion and cyclic permutation, the homotopy bands together with scalars $\lambda \in \kk^*$   and strictly positive integers $n$ determine the \emph{band complexes}  $\B_{\sigma, \lambda, n}$.  
By \cite[Thm. 3]{BM}, the string and band complexes form a complete set of indecomposable complexes in $\KminusL$ up to isomorphism. We refer to \cite{BM} for precise details of the construction of the complexes $\P_\sigma$ and $\B_{\sigma,\lambda,n}$ from a homotopy string or band $\sigma$.  The indecomposable perfect complexes are precisely the band complexes and the string complexes arising from finite homotopy strings.
 In what follows, we will only consider band complexes $\B_{\sigma, \lambda, 1}$ which we will write as $\B_{\sigma, \lambda}$. When we do not wish to explicitly identify whether an indecomposable complex is a string or band complex we use the notation
\[
\Q_\sigma =
\left\{
\begin{array}{ll}
\P_\sigma                 & \text{if $\sigma$ is a (possibly infinite) homotopy string;} \\
\B_{\sigma,\lambda} & \text{if $\sigma$ is a homotopy band.}
\end{array}
\right.
\]

Let $\sigma = \sigma_n \cdots \sigma_1$ be a homotopy string or band. We shall make frequent use of the following \emph{unfolded diagram} notation from \cite{ALP}:
\[
\xymatrix{ \xydot \arr^-{\sigma_n} & \xydot \ar@{.}[r] & \xydot \arr^-{\sigma_2} & \xydot \arr^-{\sigma_1} & \xydot }.
\]

\subsection{The standard basis}

Let $\sigma$ and $\tau$ be homotopy strings or bands. The main theorem of \cite{ALP} establishes a basis for $\Hom_{\Db(\Lambda)}(\Q_\sigma,\Q_\tau)$ which we recall here.

\begin{theorem}[{\cite[Theorem 3.15]{ALP}}] \label{thm:ALP}
Let $\sigma$ and $\tau$ be homotopy strings or bands. Then there is a basis of $\Hom_{\Db(\Lambda)}(\Q_\sigma,\Q_\tau)$ given by:
\begin{itemize}
\item graph maps $\f \colon \Q_\sigma \to \Q_\tau$;
\item singleton single maps $\f \colon \Q_\sigma \to \Q_\tau$;
\item singleton double maps $\f \colon \Q_\sigma \to \Q_\tau$;
\item quasi-graph maps $\phi \colon \Q_\sigma \rightsquigarrow \Sigma^{-1} \Q_\tau$.
\end{itemize}
\end{theorem}

A quasi-graph map is not a map but in fact determines a homotopy class of single and double maps, hence we use the notation $\rightsquigarrow$. We now briefly recap the types of maps occurring in Theorem~\ref{thm:ALP}. 

\subsubsection{Graph maps} \label{sec:graph-map-def}

Suppose $\sigma$ and $\tau$ are homotopy strings or bands of the form,
\[
\sigma = \beta \sigma_L \rho \sigma_R \alpha \text{ and } \tau = \delta \tau_L \rho \tau_R \gamma, \text{ or }
\sigma = \beta \sigma_L \rho \text{ and } \tau = \delta \tau_L \rho,
\]
where $\alpha, \beta, \gamma$ and $\delta$ are homotopy substrings, $\sigma_L, \sigma_R, \tau_L$ and $\tau_R$ 
are homotopy letters (possibly empty, in which case
the corresponding homotopy substring $\alpha, \beta, \gamma$ or $\delta$ would be empty as well), and $\rho$ is a (possibly trivial) maximal common homotopy substring, and in the second case an infinite homotopy substring of $\sigma$ and $\tau$. Moreover, we assume that $\rho$ occurs in the same cohomological degrees in both homotopy strings.
These setups are indicated in the following unfolded diagrams of $\Q_\sigma$ and $\Q_\tau$:

\begin{equation} \label{finite-graph-map}
\xymatrix@!R=4px{
\Q_\sigma \colon & \ar@{~}[r]^-{\beta} & \xydot \arr^-{\sigma_L} \ar[d]_-{f_L} \ar@{}[dr]|{(*)} & \xydot \arr^-{\rho_k} \ar@{=}[d] & \xydot \arr^-{\rho_{k-1}} \ar@{=}[d] & \cdots \arr^-{\rho_2}  & \xydot \arr^-{\rho_1} \ar@{=}[d] & \xydot \arr^-{\sigma_R} \ar@{=}[d] \ar@{}[dr]|{(**)} & \xydot \ar@{~}[r]^-{\alpha} \ar[d]^-{f_R} & \\
\Q_\tau \colon      & \ar@{~}[r]_-{\delta} & \xydot \arr_-{\tau_L}                                                & \xydot \arr_-{\rho_k}                  & \xydot \arr_-{\rho_{k-1}}                  & \cdots \arr_-{\rho_2}  & \xydot \arr_-{\rho_1}                  & \xydot \arr_-{\tau_R}                                                    & \xydot \ar@{~}[r]_-{\gamma}                 & 
}
\end{equation}
\begin{equation} \label{infinite-graph-map}
\xymatrix@!R=4px{
\P_\sigma \colon & \ar@{~}[r]^-{\beta}  & \xydot \arr^-{\sigma_L} \ar[d]_-{f_L} \ar@{}[dr]|{(*)} & \xydot \arr^-{\rho_{-1}} \ar@{=}[d]  & \xydot \arr^-{\rho_{-2}} \ar@{=}[d] & \xydot \arr^-{\rho_{-3}} \ar@{=}[d]  & \xydot \ar@{.}[r] &  \\
\P_\tau \colon      & \ar@{~}[r]_-{\delta} & \xydot \arr_-{\tau_L}                                                 & \xydot \arr_-{\rho_{-1}}                   & \xydot \arr_-{\rho_{-2}}                   & \xydot \arr_-{\rho_{-3}}                  & \xydot \ar@{.}[r] & }
\end{equation}
where we require $(*)$ and $(**)$ to commute. 
This data defines a map of complexes $\f \colon \Q_\sigma \to \Q_\tau$ called a \emph{graph map}. The isomorphisms indicated in diagrams \eqref{finite-graph-map} and \eqref{infinite-graph-map} together with $f_L$ and $f_R$, should they exist, are the \emph{components} of $\f$.

Below we explicitly write down the \emph{right endpoint conditions} making $(**)$ commute, where we have separated out the situation in \eqref{infinite-graph-map} as a kind of `infinite right endpoint condition':
\begin{align*}
\text{(RG1)} \
\xymatrix@!R=4px{
\xydot \ar[r]^-{\sigma_R} \ar@{=}[d]     & \xydot \ar@{~}[r]^-{\alpha} \ar[d]^-{f_R} & \\
\xydot \ar[r]_-{\tau_R = \sigma_R f_R}  & \xydot \ar@{~}[r]_-{\gamma}                 &
}
\quad & \text{(RG2)} \
\xymatrix@!R=4px{
\xydot \ar@{=}[d] & \ar[l]_-{\sigma_R = \bar{\omega} \bar{f}_R} \xydot \ar@{~}[r]^-{\alpha} \ar[d]^-{f_R} & \\
\xydot                   & \ar[l]^-{\tau_R = \bar{\omega}}                                            \xydot \ar@{~}[r]_-{\gamma}                 &
} \\
\text{(RG3)} \
\xymatrix@!R=4px{
\xydot \ar[r]^-{\sigma_R} \ar@{=}[d] &                        \xydot \ar@{~}[r]^-{\alpha}     & \\
\xydot                                               & \ar[l]^-{\tau_R} \xydot \ar@{~}[r]_-{\gamma}  &
}
\quad & \text{(RG$\infty$)} \
\xymatrix@!R=4px{
\xydot \arr^-{\sigma_L} \ar[d]_-{f_L} \ar@{}[dr]|{(*)} & \xydot \arr^-{\rho_{-1}} \ar@{=}[d]  & \xydot \arr^-{\rho_{-2}} \ar@{=}[d] & \xydot \arr^-{\rho_{-3}} \ar@{=}[d]  & \xydot \ar@{.}[r] &  \\
\xydot \arr_-{\tau_L}                                                 & \xydot \arr_-{\rho_{-1}}                   & \xydot \arr_-{\rho_{-2}}                   & \xydot \arr_-{\rho_{-3}}                  & \xydot \ar@{.}[r] & }
\end{align*}
The \emph{left endpoint conditions} making $(*)$ commute are defined dually.

The maximality of $\rho$ as a common homotopy substring of $\sigma$ and $\tau$ necessarily means that $\sigma_L \neq \tau_L$ and $\sigma_R \neq \tau_R$ unless, in each case, both are empty. We write $\length{\rho} = k$ for the \emph{length} of the overlap $\rho$. In the second case, we write $\length{\rho} = \infty$.
                           
In \eqref{finite-graph-map} we use the notation $\Q$ because the complexes may be either string or band complexes. As soon as the maximal common homotopy substring is infinite, however, we are forced to have string complexes arising from infinite homotopy strings. The notation $\xymatrix{ \ar@{~}[r] &}$ denotes a homotopy substring.
In the case $\sigma = \tau$ is a homotopy band then for the Hom-space $\Hom_{\Db(\Lambda)}(\B_{\sigma,\lambda},\B_{\sigma,\mu})$ there is a subtle distinction between the case $\lambda = \mu$ and $\lambda \neq \mu$. In the first case the overlap $\rho = \sigma$ gives rise to a graph map, namely the identity map. In the second case the overlap $\rho = \sigma$ does not give rise to a nonzero (graph) map. In both cases, all other overlaps $\rho$ give rise to graph maps as indicated.

\subsubsection{Single maps} \label{sec:single-def}

The unfolded diagrams of a single map $\f \colon \Q_\sigma \to \Q_\tau$ is
\[
\xymatrix@!R=4px{
\Q_\sigma :  & \ar@{~}[r]^-{\beta}  & \xydot \arr^-{\sigma_L} & \xydot \arr^-{\sigma_R} \ar[d]^-{f} & \xydot \ar@{~}[r]^-{\alpha} & \\
\Q_\tau      :  & \ar@{~}[r]_-{\delta} & \xydot \arr_-{\tau_L}     & \xydot \arr_-{\tau_R}                      & \xydot \ar@{~}[r]_-{\gamma} &
}
\]
where $f$ is a nontrivial path in $(Q,I)$ subject to the following conditions:
\begin{itemize}
\item[(L1)] either $\sigma_L$ is inverse or $\sigma_L$ is direct and $\sigma_Lf$ has a subpath in $I$;
\item[(L2)] either $\tau_L$ is  direct or $\tau_L$ is inverse and $f \bar{\tau}_L$ has a subpath in $I$;
\item[(R1)] either $\sigma_R$ is direct or $\sigma_R$ is inverse and $\bar{\sigma}_R f$ has a subpath in $I$; and
\item[(R2)] either $\tau_R$ is inverse or $\tau_R$ is direct and $f \tau_R$ has a subpath in $I$.
\end{itemize}

A single map $\f \colon \Q_\sigma \to \Q_\tau$, with single component $f$, is called a \emph{singleton single map} if its unfolded diagram, up to inversion of one of $\sigma$ or $\tau$, is
\[
\xymatrix@!R=4px{
\ar@{~}[r]^-{\beta}   & \xydot \arr^-{\sigma_L} & \xydot \ar[d]^{f} \ar[r]^-{\sigma_R = f f_R} & \xydot  \ar@{~}[r]^-{\alpha}                                                                      & \\
\ar@{~}[r]_-{\delta}  & \xydot \arr_-{\tau_L}     & \xydot                                                        & \ar[l]^-{\tau_R = \overline{f} \overline{f_L}} \xydot \ar@{~}[r]_-{\gamma} &
                }
\]
where $f$ does not contain $\sigma_L$ nor $\tau_L$ as a homotopy subletter, $\sigma_L$ and $\tau_L$ never contain $f$ as a homotopy subletter, and any of $\sigma_L$, $\sigma_R$, $\tau_L$ or $\tau_R$ is allowed to be an empty homotopy letter $\emptyset$.

Not all single maps are nonzero in $\KminusL$, those that are occur either as singleton single maps or in a homotopy class determined by a quasi-graph map, which is described below.

\subsubsection{Double maps} \label{sec:double-def}

A double map $\f \colon \Q_\sigma \to\Q_\tau$ has the following unfolded diagram:
\[
\xymatrix@!R=4px{
\Q_\sigma :  & \ar@{~}[r]^-{\beta}  & \xydot \arr^-{\sigma_L} & \xydot \ar[r]^-{\sigma_C} \ar[d]_-{f_L} & \xydot \arr^-{\sigma_R} \ar[d]^-{f_R} & \xydot \ar@{~}[r]^-{\alpha} &  \\
\Q_\tau      :  & \ar@{~}[r]_-{\delta} & \xydot \arr_-{\tau_L}     & \xydot \ar[r]_-{\tau_C}                         & \xydot \arr_-{\tau_R}                          & \xydot \ar@{~}[r]_-{\gamma} &
}
\]
where $f_L$ and $f_R$ are nontrivial paths in $(Q,I)$ such that $f_L \tau_C = \sigma_C f_R$ has no subpath in $I$, conditions (L1) and (L2) hold for $f_L$ and (R1) and (R2) hold for $f_R$.
A double map $\f \colon \Q_\sigma \to \Q_\tau$ is a \emph{singleton double map} if there exists a nontrivial path $f'$ in $(Q,I)$ such that $\sigma_C = f_L f'$ and $\tau_C = f' f_R$.

As above, not all double maps are nonzero in $\KminusL$, those that are occur either as singleton double maps or in a homotopy class determined by a quasi-graph map.

\subsubsection{Quasi-graph maps} \label{sec:quasi-graph}

If in Section~\ref{sec:graph-map-def}\eqref{finite-graph-map}, the squares marked $(*)$ and $(**)$ do not commute, then the diagram \eqref{finite-graph-map} determines a \emph{quasi-graph map} $\varphi \colon \Q_\sigma \rightsquigarrow \Q_\tau$. Below we spell out the non-commuting endpoint conditions for $(**)$, which we call the \emph{quasi-graph map right endpoint conditions},
\[
\text{(RQ1)} \
\xymatrix@!R=4px{
\xydot \ar[r]^-{\sigma_R=\tau_R \tau'_R} \ar@{=}[d]  & \xydot \ar@{~}[r]^-{\alpha}     & \\
\xydot \ar[r]_-{\tau_R}                                                 & \xydot \ar@{~}[r]_-{\gamma}  &
}
\quad  \text{(RQ2)} \
\xymatrix@!R=4px{
\xydot \ar@{=}[d] & \ar[l]_-{\sigma_R = \bar{\omega}} \xydot \ar@{~}[r]^-{\alpha} & \\
\xydot                   & \ar[l]^-{\tau_R = \bar{\omega} \bar{\omega'}}                                            \xydot \ar@{~}[r]_-{\gamma}                 &
}
\quad \text{(RQ3)} \
\xymatrix@!R=4px{
\xydot \ar@{<-}[r]^-{\sigma_R} \ar@{=}[d] &                        \xydot \ar@{~}[r]^-{\alpha}     & \\
\xydot                                               & \ar@{<-}[l]^-{\tau_R} \xydot \ar@{~}[r]_-{\gamma}  &
}
\]
where $\omega$ and $\omega'$ are direct homotopy letters, and all primed homotopy (sub)letters are nontrivial. In (RQ3) we also permit $\sigma_R$ or $\tau_R$ (but not both) to be zero.
The \emph{quasi-graph map left endpoint conditions} are defined dually.

A quasi-graph map $\varphi \colon \Q_\sigma \rightsquigarrow \Sigma^{-1} \Q_\tau$ determines a family of homotopy equivalent single and/or double maps. All nonzero single or double maps that are not singleton occur in this way. We refer to \cite[Prop. 4.8]{ALP} for explicit details on the single or double maps lying in a homotopy class determined by a given quasi-graph map.

Finally, we remark `quasi-graph maps' satisfying the graph map endpoint conditions give rise to a family of null-homotopic single and/or double maps. As such quasi-graph maps never arise from diagrams of the form \eqref{infinite-graph-map} because such diagrams always satisfy a graph map endpoint conditions, namely (RG$\infty$) or its dual (LG$\infty$). 

\section{Mapping cones of graph maps} \label{sec:graph-maps}

In this section we compute the mapping cones of graph maps. We first start by describing the unfolded diagram of the mapping cone of a graph map.

\subsection{The unfolded diagram of the mapping cone of a graph map} \label{sec:graph-mc-unfolded}

Let $\f \colon \Q_\sigma \to \Q_\tau$ be a graph map. The unfolded diagram of $\M_{\f}$ has one of the following forms.
\[
\xymatrix@!R=4px{
\ar@{~}[r]^-{-\beta} & \xydot \arr^-{-\sigma_L} \ar[dr]_<<<<{f_L}  & \xydot \arr^-{-\rho_k} \ar@{=}[dr] & \xydot \arr^-{-\rho_{k-1}} \ar@{=}[dr] & \cdots \arr^-{-\rho_2}  & \xydot \arr^-{-\rho_1} \ar@{=}[dr] & \xydot \arr^-{-\sigma_R} \ar@{=}[dr]  & \xydot \ar@{~}[r]^-{-\alpha} \ar[dr]^>>>>{f_R} & & \\
 & \ar@{~}[r]_-{\delta} & \xydot \arr_-{\tau_L}                                                & \xydot \arr_-{\rho_k}                  & \xydot \arr_-{\rho_{k-1}}                  & \cdots \arr_-{\rho_2}  & \xydot \arr_-{\rho_1}                  & \xydot \arr_-{\tau_R}                                                    & \xydot \ar@{~}[r]_-{\gamma}                 & 
}
\]
\[
\xymatrix@!R=4px{
 \ar@{.}[r] & \xydot \arr^-{-\rho_3} \ar@{=}[dr] & \xydot \arr^-{-\rho_2} \ar@{=}[dr] & \xydot \arr^-{-\rho_1} \ar@{=}[dr] & \xydot \arr^-{-\sigma_R} \ar@{=}[dr]  & \xydot \ar@{~}[r]^-{-\alpha} \ar[dr]^>>>>{f_R} & & \\
 & \ar@{.}[r] & \xydot \arr_-{\rho_3}                 & \xydot \arr_-{\rho_2}                   & \xydot \arr_-{\rho_1}                  & \xydot \arr_-{\tau_R}                                                & \xydot \ar@{~}[r]_-{\gamma}                 & 
}
\]

\subsection{Mapping cones of graph maps between string complexes}

Let $\sigma$ and $\tau$ be homotopy strings and $\f \colon \P_\sigma \to \P_\tau$ be a graph map between the corresponding string complexes. In this case the generic situation is for $\M_{\f}$ to have two indecomposable summands, which are again string complexes and whose strings can be read off from $\sigma$, $\tau$ and $\f$ in a natural way. We start with a technical definition which ensures that the unfolded diagram of the graph map is properly oriented for the mapping cone calculus in the case that $\f$ is a graph map supported in precisely one degree.

\begin{definition} \label{def:graph-compatible}
Let $\sigma$ and $\tau$ be homotopy strings and suppose $\f \colon \P_\sigma \to \P_\tau$ is a graph map concentrated in precisely one degree, i.e. the maximal common homotopy substring $\rho$ is trivial and satisfies the graph map endpoint conditions (RG3) and (LG3):
\[
\xymatrix@!R=4px{
\ar@{~}[r]^-{\beta}  & \xydot                         & \ar[l]_-{\sigma_L} \xydot \ar@{=}[d] \ar[r]^-{\sigma_R} & \xydot \ar@{~}[r]^-{\alpha}                             &  \\
\ar@{~}[r]_-{\delta} & \xydot \ar[r]_-{\tau_L} & \xydot                                                                           & \ar[l]^-{\tau_R} \xydot \ar@{~}[r]_-{\gamma} & 
}
\]
We say that the homotopy strings $\sigma$ and $\tau$ are \emph{compatibly oriented} if $\sigma_L \tau_L \neq 0$ and $\sigma_R \tau_R \neq 0$.
\end{definition}

We observe that if a graph map is supported in more than one degree, either by having precisely one isomorphism and satisfying endpoint conditions so that one of $f_L$ or $f_R$ is nonempty, or by having isomorphisms in at least two cohomological degrees (making $\rho$ of length at least $1$), then the homotopy strings and graph map between them is forced to be `compatibly oriented'.

Note that in the following if a homotopy string $\sigma = \emptyset$ then the corresponding string complex $\P_\sigma \cong 0^\bullet$.

\begin{theorem} \label{thm:graph-maps}
Let $\sigma$ and $\tau$ be (possibly infinite) homotopy strings and suppose $\f \colon \P_\sigma \to \P_\tau$ is a graph map. Suppose that  $\sigma$ and $\tau$ are compatibly oriented. Then we have the following cases.
\begin{enumerate}
\item If $\sigma = \beta \sigma_L \rho \sigma_R \alpha$ and $\tau = \delta \tau_L \rho \tau_R \gamma$ then the mapping cone $\M_{\f}$ is isomorphic in $\KminusL$ to the direct sum of $\P_{c_1} \oplus \P_{c_2}$, where
\[
c_1 = 
\left\{
\begin{array}{ll}
\bar{\gamma} \bar{\tau}_R \sigma_R \alpha & \text{ if } \sigma_R \neq \emptyset \text{ and } \tau_R \neq \emptyset; \\
\bar{\gamma}                              & \text{ if } \sigma_R = \emptyset; \\
\alpha                                    & \text{ if } \tau_R = \emptyset; \\
\emptyset                               & \text{ if } \sigma_R = \emptyset \text{ and } \tau_R = \emptyset,
\end{array}
\right.
\text{ and } 
c_2 = 
\left\{
\begin{array}{ll}
\beta \sigma_L \bar{\tau}_L \bar{\delta} & \text{ if } \sigma_L \neq \emptyset \text{ and } \tau_L \neq \emptyset; \\
\bar{\delta}                                             & \text{ if } \sigma_L = \emptyset; \\
\beta                                                      & \text{ if } \tau_L = \emptyset; \\
\emptyset                                            & \text{ if } \sigma_L = \emptyset \text{ and } \tau_L = \emptyset,
\end{array}
\right.
\]
\item If  $\sigma = \rho \sigma_R \alpha$ and $\tau = \rho \tau_R \gamma$ then the mapping cone $\M_{\f}$ is isomorphic in $\KminusL$ to $\P_c$, where 
\[
c = 
\left\{
\begin{array}{ll} 
\bar{\gamma} \bar{\tau}_R \sigma_R \alpha & \text{ if } \sigma_R \neq \emptyset \text{ and } \tau_R \neq \emptyset; \\
\bar{\gamma}                                                & \text{ if } \sigma_R = \emptyset; \\
\alpha                                                            & \text{ if } \tau_R = \emptyset; \\
\emptyset                                                    & \text{ if } \sigma_R = \emptyset \text{ and } \tau_R = \emptyset.
\end{array}
\right.
\]
\end{enumerate}
\end{theorem}

\begin{remark} \label{rem:endpoints}
If in Theorem~\ref{thm:graph-maps}(1) the graph map right endpoint conditions (RG1) or (RG2) occur then in either case $\bar{\gamma} \bar{\tau}_R \sigma_R \alpha = \bar{\gamma} \bar{f}_R \alpha$. Similarly, if the left endpoint conditions (LG1) or (LG2) occur then $\beta \sigma_L \bar{\tau}_L \bar{\delta}=\beta f_L \bar{\delta}$.
\end{remark}

In order to prove Theorem~\ref{thm:graph-maps} we first need the following well-known result. For the convenience of the reader we provide a short proof.

\begin{lemma} \label{lem:homotopy}
Let $\P$ be a complex of the form
\[
\xymatrix{
\P \colon & \cdots \ar[r] & P^{n-1} \ar[r]^-{\begin{pmat} a_1 & a_2 \end{pmat}} & P^n \oplus Q \ar[r]^-{\begin{pmat} b_1 & b _2 \\ b_3 & 1 \end{pmat}} &  P^{n+1} \oplus Q \ar[r]^-{\begin{pmat} c_1 \\ c_2 \end{pmat}} &  P^{n+2} \ar[r] & \cdots.
}
\] 
Then $\P \cong (P')^\bullet \oplus \Q$, where $(P')^\bullet$ and $\Q$ are the following complexes,
\begin{align*}
(P')^\bullet \colon & \xymatrix{
\cdots \ar[r] & P^{n-1} \ar[r]^-{a_1} & P^n \ar[r]^-{b_1 - b_2 b_3} & P^{n+1} \ar[r]^-{c_1} & P^{n+2} \ar[r] & \cdots
}; \\
\Q \colon & \xymatrix{
\cdots \ar[r] & 0 \ar[r] & Q \ar[r]^-{1} & Q \ar[r] & 0 \ar[r] & \cdots
},
\end{align*} 
where $(P')^i = P^i$ and $Q^i = 0$ whenever $i \notin \{n, n+1\}$. In particular, in the homotopy category $\P \cong (P')^\bullet$. 
\end{lemma}

\begin{proof}
The required isomorphism (and its inverse) are recorded in the following commutative diagram of complexes.
\[
\xymatrix@!C=5.5pc{
\P \colon \ar[d]_-{\sim}                    & P^{n-1} \ar[r]^-{\begin{pmat} a_1 & a_2 \end{pmat}} \ar@{=}[d] & P^n \oplus Q \ar[r]^-{\begin{pmat} b_1 & b _2 \\ b_3 & 1 \end{pmat}} \ar[d]_-{\begin{pmat} 1 & b_2 \\ 0 & 1 \end{pmat}}        &  P^{n+1} \oplus Q \ar[r]^-{\begin{pmat} c_1 \\ c_2 \end{pmat}} \ar[d]^-{\begin{pmat} 1 & 0 \\ -b_3 & 1 \end{pmat}} &  P^{n+2} \ar@{=}[d]  \\
(P')^\bullet \oplus \Q \colon \ar[d]_-{\sim} & P^{n-1} \ar[r]_-{\begin{pmat} a_1 & 0 \end{pmat}} \ar@{=}[d]   & P^n \oplus Q \ar[r]^-{\begin{pmat} b_1 - b_2 b_3 & 0 \\ 0 & 1 \end{pmat}} \ar[d]_-{\begin{pmat} 1 & -b_2 \\ 0 & 1 \end{pmat}}  &  P^{n+1} \oplus Q \ar[r]^-{\begin{pmat} c_1 \\ 0 \end{pmat}} \ar[d]^-{\begin{pmat} 1 & 0 \\ b_3 & 1 \end{pmat}}    &  P^{n+2} \ar@{=}[d] \\
\P \colon                                   & P^{n-1} \ar[r]_-{\begin{pmat} a_1 & a_2 \end{pmat}}           & P^n \oplus Q \ar[r]_-{\begin{pmat} b_1 & b _2 \\ b_3 & 1 \end{pmat}}                                                           &  P^{n+1} \oplus Q \ar[r]_-{\begin{pmat} c_1 \\ c_2 \end{pmat}}                                                     &  P^{n+2}   
}
\]
Note that the commutativity of each square follows from observing that $a_2 = -a_1 b_2$ and $c_2 = - b_3 c_1$, which follows from the fact that $\P$ is a complex so that the differential squares to zero. The final statement follows by observing that $\Q$ is a contractible complex.
\end{proof}

\begin{remark}[Unfolded diagram interpretation of Lemma~\ref{lem:homotopy}]
We can reformulate the situation of Lemma~\ref{lem:homotopy} in terms of unfolded diagrams as follows. The initial complex $\P$ has the following `unfolded diagram':
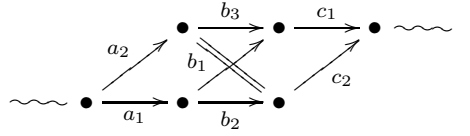
\[
\xymatrix@!R=4pt{
           &                                     & \xydot \ar[r]^-{b_3} \ar@{=}[dr]   & \xydot \ar[r]^-{c_1}  & \xydot \ar@{~}[r] & \, . \\
\ar@{~}[r] & \xydot \ar[r]_-{a_1} \ar[ur]^-{a_2}  & \xydot \ar[ur]^<<<{b_1} \ar[r]_-{b_2} & \xydot \ar[ur]_-{c_2} &                  &  }
\]
The complex $(P')^\bullet$, to which $\P$ is homotopic, has the following `unfolded diagram':
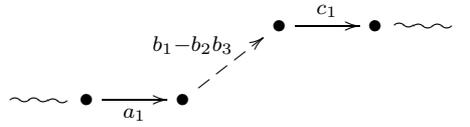
\[
\xymatrix@!R=4pt{
           &                       &                                    & \xydot \ar[r]^-{c_1}  & \xydot \ar@{~}[r] & \, . \\
\ar@{~}[r] & \xydot \ar[r]_-{a_1}  & \xydot \ar@{-->}[ur]^-{b_1-b_2 b_3}  &                      &                  &  }
\]
\end{remark}

Let $\sigma$ and $\tau$ be (possibly infinite) homotopy strings and suppose $\f \colon \P_\sigma \to \P_\tau$ is a graph map. Since $\f$ is a graph map, at least one nonzero component of $\f$ is an isomorphism. The next lemma summarises a straightforward analysis of the form of an unfolded diagram of $\f$ at one such component. 

\begin{lemma} \label{lem:six-cases}
Let $\sigma$ and $\tau$ be (possibly infinite) homotopy strings and suppose $\f \colon \P_\sigma \to \P_\tau$ is a graph map. At any given component that is an isomorphism, the unfolded diagram of $\f$ locally has the one of the following six forms, up to inversion of $\sigma$, $\tau$ or both:
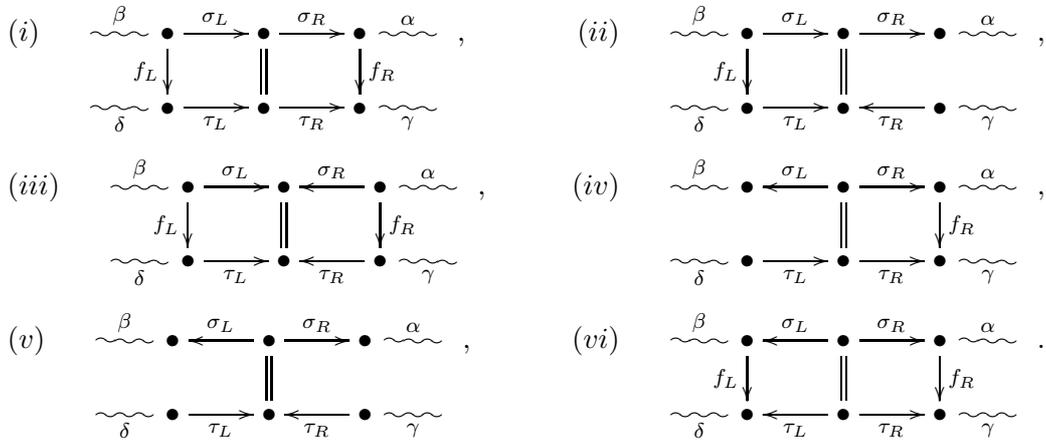
\begin{align*}
& (i) \quad
\xymatrix@!R=4pt{
\ar@{~}[r]^-{\beta}  & \xydot \ar[r]^-{\sigma_L} \ar[d]_-{f_L} & \xydot \ar[r]^-{\sigma_R} \ar@{=}[d] & \xydot \ar[d]^-{f_R} \ar@{~}[r]^-{\alpha} & \\
\ar@{~}[r]_-{\delta} & \xydot \ar[r]_-{\tau_L}                 & \xydot \ar[r]_-{\tau_R}             & \xydot \ar@{~}[r]_-{\gamma}               &
},
& (ii) \quad
\xymatrix@!R=4pt{
\ar@{~}[r]^-{\beta}  & \xydot \ar[r]^-{\sigma_L} \ar[d]_-{f_L} & \xydot \ar[r]^-{\sigma_R} \ar@{=}[d] & \xydot \ar@{~}[r]^-{\alpha}  & \\
\ar@{~}[r]_-{\delta} & \xydot \ar[r]_-{\tau_L}                 & \xydot \ar@{<-}[r]_-{\tau_R}         & \xydot \ar@{~}[r]_-{\gamma} &
}, \\
& (iii) \quad
\xymatrix@!R=4pt{
\ar@{~}[r]^-{\beta}  & \xydot \ar[r]^-{\sigma_L} \ar[d]_-{f_L} & \xydot \ar@{<-}[r]^-{\sigma_R} \ar@{=}[d] & \xydot \ar[d]^-{f_R} \ar@{~}[r]^-{\alpha} & \\
\ar@{~}[r]_-{\delta} & \xydot \ar[r]_-{\tau_L}                 & \xydot \ar@{<-}[r]_-{\tau_R}             & \xydot \ar@{~}[r]_-{\gamma}               &
},
& (iv) \quad
\xymatrix@!R=4pt{
\ar@{~}[r]^-{\beta}  & \xydot \ar@{<-}[r]^-{\sigma_L}  & \xydot \ar[r]^-{\sigma_R} \ar@{=}[d] & \xydot \ar[d]^-{f_R} \ar@{~}[r]^-{\alpha} & \\
\ar@{~}[r]_-{\delta} & \xydot \ar[r]_-{\tau_L}         & \xydot \ar[r]_-{\tau_R}             & \xydot \ar@{~}[r]_-{\gamma}               &
}, \\
& (v) \quad
\xymatrix@!R=4pt{
\ar@{~}[r]^-{\beta}  & \xydot \ar@{<-}[r]^-{\sigma_L}  & \xydot \ar[r]^-{\sigma_R} \ar@{=}[d] & \xydot \ar@{~}[r]^-{\alpha} & \\
\ar@{~}[r]_-{\delta} & \xydot \ar[r]_-{\tau_L}         & \xydot \ar@{<-}[r]_-{\tau_R}         & \xydot \ar@{~}[r]_-{\gamma} &
},
& (vi) \quad
\xymatrix@!R=4pt{
\ar@{~}[r]^-{\beta}  & \xydot \ar@{<-}[r]^-{\sigma_L} \ar[d]_-{f_L} & \xydot \ar[r]^-{\sigma_R} \ar@{=}[d] & \xydot \ar[d]^-{f_R} \ar@{~}[r]^-{\alpha} & \\
\ar@{~}[r]_-{\delta} & \xydot \ar@{<-}[r]_-{\tau_L}                 & \xydot \ar[r]_-{\tau_R}             & \xydot \ar@{~}[r]_-{\gamma}               &
}.
\end{align*}
Here $\sigma_L$, $\sigma_R$, $\tau_L$ and $\tau_R$ need not be located at the endpoints of overlaps and may be empty paths. The homotopy letters $f_L$, $f_R$ may be paths, trivial paths or empty paths.
\end{lemma}

Recall the form of the unfolded diagram of $\M_{\f}$ from Section~\ref{sec:graph-mc-unfolded}.
The following corollary applies the unfolded diagram interpretation of Lemma~\ref{lem:homotopy} to each of the local situations occurring in Lemma~\ref{lem:six-cases}. The resulting unfolded diagram gives rise to a complex that is homotopic to $\M_{\f}$.

\begin{corollary} \label{cor:homotopy} 
Let $\sigma$ and $\tau$ be (possibly infinite) homotopy strings and suppose $\f \colon \P_\sigma \to \P_\tau$ is a graph map.
For each of the cases in Lemma~\ref{lem:six-cases} above, the following (compatibly oriented) unfolded diagrams correspond to complexes that are homotopic to $\M_{\f}$, where we use the notation of Lemma~\ref{lem:six-cases}.
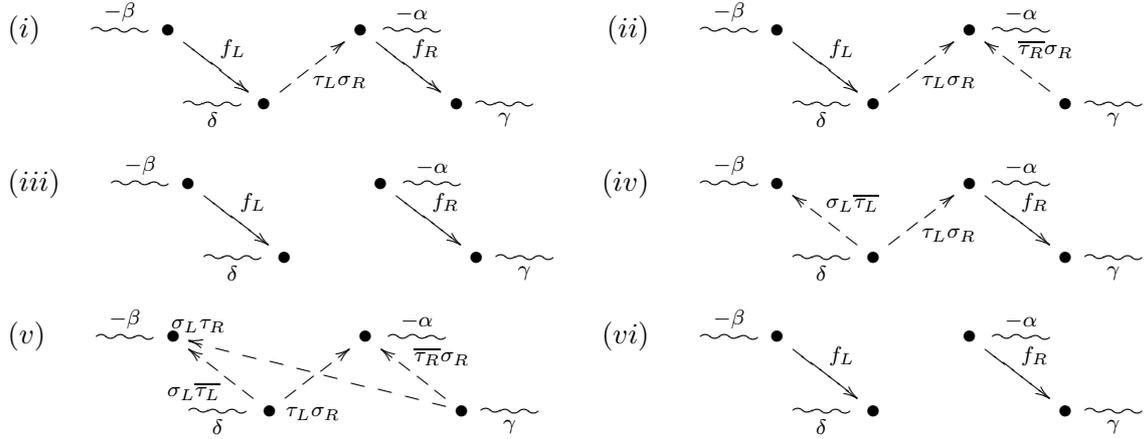
\begin{align*}
& (i) \quad
\xymatrix@!R=4pt{
\ar@{~}[r]^-{-\beta} & \xydot \ar[dr]^-{f_L} &                                         & \xydot \ar[dr]^-{f_R} \ar@{~}[r]^-{-\alpha} &                             & \\
                    & \ar@{~}[r]_-{\delta}      & \xydot \ar@{-->}[ur]_-{\tau_L \sigma_R} &                                            & \xydot \ar@{~}[r]_-{\gamma} &
}
& (ii) \quad
\xymatrix@!R=4pt{
\ar@{~}[r]^-{-\beta} & \xydot \ar[dr]^-{f_L} &                                         & \xydot  \ar@{~}[r]^-{-\alpha} &                                                                  & \\
                    & \ar@{~}[r]_-{\delta}      & \xydot \ar@{-->}[ur]_-{\tau_L \sigma_R} &                              & \ar@{-->}[ul]_-{\overline{\tau_R}\sigma_R} \xydot \ar@{~}[r]_-{\gamma} &
} \\
& (iii) \quad
\xymatrix@!R=4pt{
\ar@{~}[r]^-{-\beta} & \xydot \ar[dr]^-{f_L} &         & \xydot \ar[dr]^-{f_R} \ar@{~}[r]^-{-\alpha} &                             & \\
                    & \ar@{~}[r]_-{\delta}   & \xydot  &                                            & \xydot \ar@{~}[r]_-{\gamma} &
} 
& (iv) \quad
\xymatrix@!R=4pt{
\ar@{~}[r]^-{-\beta} & \xydot           &                                                                               & \xydot \ar[dr]^-{f_R} \ar@{~}[r]^-{-\alpha} &                             & \\
                    & \ar@{~}[r]_-{\delta} & \ar@{-->}[ul]_-{\sigma_L \overline{\tau_L}} \xydot \ar@{-->}[ur]_-{\tau_L \sigma_R} &                                            & \xydot \ar@{~}[r]_-{\gamma} &
} \\
& (v) \quad
\xymatrix@!R=4pt{
\ar@{~}[r]^-{-\beta} & \xydot           &                                                                               & \xydot \ar@{~}[r]^-{-\alpha} &                             & \\
                    & \ar@{~}[r]_-{\delta} & \ar@{-->}[ul]^-{\sigma_L \overline{\tau_L}} \xydot \ar@{-->}[ur]_<{\tau_L \sigma_R} &                             & \ar@{-->}[ulll]_>{\sigma_L \tau_R} \ar@{-->}[ul]_-{\overline{\tau_R} \sigma_R} \xydot \ar@{~}[r]_-{\gamma} &
}
& (vi) \quad
\xymatrix@!R=4pt{
\ar@{~}[r]^-{-\beta} & \xydot \ar[dr]^-{f_L} &        & \xydot \ar[dr]^-{f_R} \ar@{~}[r]^-{-\alpha} &                             & \\
                    & \ar@{~}[r]_-{\delta}   & \xydot &                                            & \xydot \ar@{~}[r]_-{\gamma} &
}
\end{align*}
Here the dashed arrows indicate components of the map $b_1 - b_2 b_3$ in Lemma~\ref{lem:homotopy}. Note that in each case $b_1 = 0$.
\end{corollary}

\begin{remark}
Note that in cases (iii) and (iv) of Corollary~\ref{cor:homotopy} the compononents $b_2$ and $b_3$ are also zero, hence why no dashed arrows appear in those diagrams. 
\end{remark}

We are now ready for the proof of Theorem~\ref{thm:graph-maps}.

\begin{proof}[Proof of Theorem~\ref{thm:graph-maps}]
For convenience we distinguish between three cases: $\length{\rho} = 1$, $\length{\rho} = k > 1$ and $\length{\rho} = \infty$. The first two cases correspond to statement (1) of the theorem.

Lemma~\ref{lem:six-cases} shows all possible forms, up to inversion and inclusion of empty homotopy letters, in the case $\length{\rho} = 1$. Observe that the right and left graph map endpoint conditions together with compatible orientation in case (v) (see Definition~\ref{def:graph-compatible}) force, in each case, $\tau_L \sigma_R = 0$ and  $\sigma_L \overline{ \tau_L} = 0$. In particular, the diagrams (i) -- (vi) in Corollary~\ref{cor:homotopy} give precisely the unfolded diagrams of the string complexes whose homotopy strings are $c_1$ and $c_2$.
It follows that $\M_{\f} \simeq \P_{c_1} \oplus \P_{c_2}$ in this case.

Now suppose $\length{\rho} = k > 1$. Firstly consider the case $\rho_1$ is direct. 
Two applications of Corollary~\ref{cor:homotopy} show that the mapping cone $\M_{\f}$ has an unfolded diagram of the following form. On the left hand side we are in one of cases (i), (iii), (iv) or (vi) of Corollary~\ref{cor:homotopy}, and on the right hand side we are in cases (i), (ii) or (iii).
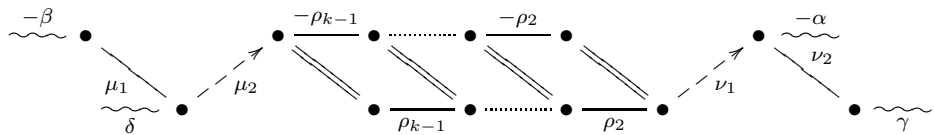
\[
\xymatrix@!R=4pt{
\ar@{~}[r]^-{-\beta} & \xydot \ar@{-}[dr]_-{\mu_1}    &                      & \xydot \ar@{=}[dr] \arr^-{-\rho_{k-1}} & \xydot \ar@{=}[dr] \ar@{..}[r] & \xydot \ar@{=}[dr] \arr^-{-\rho_2} & \xydot \ar@{=}[dr]   &                      & \xydot \ar@{~}[r]^-{-\alpha} \ar@{-}[dr]^-{\nu_2} &                             & \\
                     & \ar@{~}[r]_-{\delta}  & \xydot \ar@{-->}[ur]_-{\mu_2} &                                       & \xydot \arr_-{\rho_{k-1}}       & \xydot \ar@{..}[r]                & \xydot \arr_-{\rho_2} & \xydot \ar@{-->}[ur]_-{\nu_1} &                                       &  \xydot \ar@{~}[r]_-{\gamma} & 
} 
\]
Here, $\mu_1 = f_L$ is a downward pointing direct homotopy letter whenever $\sigma_L$ and $\tau_L$ have the same orientation (cases (i), (iii) and (vi)) and $\mu_1 = \sigma_L \overline{\tau_L}$ is an upward pointing inverse homotopy letter whenever $\sigma_L$ and $\tau_L$ have opposite orientation (case (iv)). Note that, in this case $\sigma_L \overline{\tau_L} \neq 0$.
In cases (iii) and (vi) $\mu_2$ does not exist for degree reasons and in cases (i) and (iv), $\mu_2 = \tau_L \rho_k$, which as above is zero.
Now we turn to the right hand side of the diagram. In the case that $\sigma_R$ is direct, whence we fall into (i) or (ii), $\nu_1 = \rho_1 \sigma_R = 0$. In the case that $\sigma_R$ is inverse, putting us in case (iii), then $\nu_1$ does not exist for degree reasons. Finally, if $\sigma_R$ and $\tau_R$ have the same orientation (cases (i) and (iii)), we have $\nu_2 = f_R$ is a downward pointing direct homotopy letter. If $\sigma_R$ and $\tau_R$ have opposite orientations (case (ii)) then $\nu_2 = \overline{\sigma_R} \tau_R$ is an upward pointing inverse homotopy letter, which, as before, is nonzero. 

Now we consider the case that $\rho_1$ is inverse. We have already dealt with the left hand side of the diagram above, so we only consider the right hand side of the unfolded diagram. Here we fall into case (vi) or the duals to cases (i) and (iv).
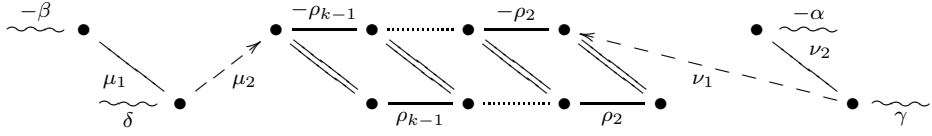
\[
\xymatrix@!R=4pt{
\ar@{~}[r]^-{-\beta} & \xydot \ar@{-}[dr]_-{\mu_1}    &                      & \xydot \ar@{=}[dr] \arr^-{-\rho_{k-1}} & \xydot \ar@{=}[dr] \ar@{..}[r] & \xydot \ar@{=}[dr] \arr^-{-\rho_2} & \xydot \ar@{=}[dr]   &                      & \xydot \ar@{~}[r]^-{-\alpha} \ar@{-}[dr]^-{\nu_2} &                             & \\
                     & \ar@{~}[r]_-{\delta}  & \xydot \ar@{-->}[ur]_-{\mu_2} &                                       & \xydot \arr_-{\rho_{k-1}}       & \xydot \ar@{..}[r]                & \xydot \arr_-{\rho_2} & \xydot  &                                       & \ar@{-->}[ulll]^-{\nu_1} \xydot \ar@{~}[r]_-{\gamma} & 
} 
\]
Here, if both $\sigma_R$ and $\tau_R$ are direct (putting us in case (vi)), $\nu_1$ does not exist for degree reasons and $\nu_2 = f_R$ is a downward pointing direct homotopy letter. If $\tau_R$ is inverse, we fall into the dual of case (i) or case (iv) and $\nu_1 = \rho_1 \tau_R = 0$. If both $\sigma_R$ and $\tau_R$ are inverse (putting us in the dual of case (i)), then $\nu_2 = f_R$ is a downward pointing direct homotopy letter. If $\sigma_R$ is direct and $\tau_R$ is inverse (dual to case (iv)), then $\nu_2 = \overline{\sigma_R}\tau_R$ is an upwards pointing inverse homotopy letter.

It follows that $\M_{\f}$ is homotopic to the direct sum of the string complexes $\P_{c_1}$ and $\P_{c_2}$ and the mapping cone of $\id \colon \P_{\rho'} \to \P_{\rho'}$ where $\rho'$ is the homotopy string $\rho' = \rho_{k-1} \cdots \rho_2$. In particular, $\M_{\id}$ is contractible.
Note that in the case $n=1$, $\rho' = \emptyset$ is the empty homotopy string and therefore $\P_{\rho'} = 0^\bullet$ and when $n = 2$,  $\rho' = 1_{e(\rho_1)}$ is a trivial homotopy string.
It follows that $\M_{\f} \simeq \P_{c_1} \oplus \P_{c_2}$.

Finally, suppose $\length{\rho} = \infty$, which puts us into statement (2) of the theorem, or its dual. Let's consider the setup in statement (2), which gives us the following unfolded diagram after one application of Corollary~\ref{cor:homotopy}, where $\mu_1$ and $\mu_2$ are defined as above.
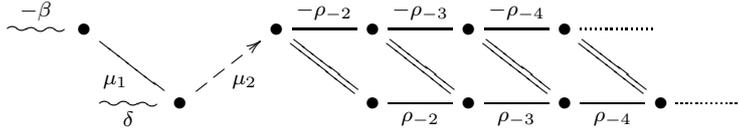
\[
\xymatrix@!R=4pt{
\ar@{~}[r]^-{-\beta} & \xydot \ar@{-}[dr]_-{\mu_1}    &                      & \xydot \ar@{=}[dr] \arr^-{-\rho_{-2}} & \xydot \ar@{=}[dr] \arr^-{-\rho_{-3}} & \xydot \ar@{=}[dr] \arr^-{-\rho_{-4}} & \xydot \ar@{=}[dr] \ar@{..}[r]  &                      &  \\
                     & \ar@{~}[r]_-{\delta}  & \xydot \ar@{-->}[ur]_-{\mu_2} &                                       & \xydot \arr_-{\rho_{-2}}       & \xydot \arr_-{\rho_{-3}}                & \xydot \arr_-{\rho_{-4}} & \xydot \ar@{..}[r]  &                                      
} 
\] 
It follows that $\M_{\f}$ is homotopic to the direct sum of the string complex $\P_c$ and the mapping cone of $\id \colon \P_{\rho'} \to \P_{\rho'}$, where $\rho'$ is the infinite homotopy string $\rho' = \rho_{-2} \rho_{-3} \rho_{-4} \cdots$. Hence, $\M_{\f} \simeq \P_c$.  
\end{proof}

\subsection{Mapping cones of graph maps involving a band complex} \label{sec:graph-band}

Suppose $\sigma$ and $\tau$ are homotopy strings or bands, with at least one being a homotopy band. In this section we consider the mapping cones of graph maps $\f \colon \Q_\sigma \to \Q_\tau$. The difference with Theorem~\ref{thm:graph-maps} is that now $\M_{\f}$ has only one indecomposable summand. Moreover, this summand is a band complex precisely when both $\sigma$ and $\tau$ are homotopy bands, and is a string complex otherwise. 

We start with the situation that both $\sigma$ and $\tau$ are homotopy bands. We impose the convention that the scalars $\lambda$ and $\mu$ are placed on direct arrows of $\sigma$ and $\tau$, respectively.

\begin{proposition} \label{prop:graph-band-to-band}
Let $\sigma = \beta \sigma_L \rho \sigma_R \alpha$ and $\tau = \delta \tau_L \rho \tau_R \gamma$ be homotopy bands. Suppose $\f \colon \B_{\sigma,\lambda} \to \B_{\tau, \mu}$ is a compatibly oriented graph map which is not the identity. Then 
\[
\M_{\f} \cong
\left\{
\begin{array}{ll}
\B_{c, \lambda \mu^{-1}}  & \text{if $\length{\rho}$ is even;} \\
\B_{c, -\lambda \mu^{-1}} & \text{if $\length{\rho}$ is odd,}
\end{array}
\right.
\]
where $c = \beta \sigma_L \bar{\tau}_L \bar{\delta} \bar{\gamma} \bar{\tau}_R \sigma_R \alpha$ and where the scalar $\lambda \mu^{-1}$ is placed on a direct homotopy letter of $\alpha$, $\beta$, $\bar{\gamma}$ or $\bar{\delta}$\footnote{It is possible here that the word $c$ is a nontrivial power of band. In this case, the mapping decomposes into a number of indecomposable summands. A corrected and extended statement (to the case that $\rho$ is longer than at least one of the homotopy bands) is given in \cite[Thm. 2.6]{Addendum}}.
\end{proposition}

Note that here we make use Remark~\ref{rem:endpoints} when describing the homotopy band $c$.

\begin{proof}
We first check that $c = \beta \sigma_L \bar{\tau}_L \bar{\delta} \bar{\gamma} \bar{\tau}_R \sigma_R \alpha$ is a homotopy band. To do this, we introduce some notation. Let
\[
\partial (\sigma) =  \, \# \,  \text{direct homotopy letters of } \sigma, \text{ and, }
\iota (\sigma) =      \, \# \, \text{inverse homotopy letters of } \sigma.
\]
We must show that $\partial (c) = \iota (c)$. Suppose, as usual, that $\sigma = \sigma_m \cdots \sigma_1$ and $\tau = \tau_n \cdots \tau_1$. Then we have $m = 2m'$ and $n = 2n'$, where $m' = \partial (\sigma) = \iota (\sigma)$ and $n' = \partial (\tau) = \iota (\tau)$.

We start by dealing with the case that $\f \colon \B_{\sigma,\lambda} \to \B_{\tau,\mu}$ does not satisfy the graph map endpoint conditions (RG3) or (LG3); note that $\f$ cannot satisfy (RG$\infty$) or (LG$\infty$) because $\sigma$ and $\tau$ are homotopy bands. In this case $c = \beta f_L \bar{\delta} \bar{\gamma} \bar{f}_R \alpha$.

Suppose $\partial (\sigma_L \rho \sigma_R ) = l = \partial (\tau_L \rho \tau_R)$ and $\iota (\sigma_L \rho \sigma_R ) = l' = \iota (\tau_L \rho \tau_R)$. Then $\partial(\alpha) + \partial(\beta) = m' - l$ and $\iota(\alpha) + \iota(\beta) = m' - l'$.  Similarly, $\partial(\gamma) + \partial(\delta) = n' - l$ and $\iota(\gamma) + \iota(\delta) = n' - l'$. It follows that
\begin{eqnarray*}
\partial (c) & = & \partial (\alpha) + \partial (\beta) + \iota (\gamma) + \iota (\delta) + 1 = m' -l + n' - l' + 1, \text{ and,} \\
\iota (c)     & = & \iota (\alpha) + \iota (\beta) + \partial (\gamma) + \partial (\delta) + 1 = m' -l' + n' - l + 1.
\end{eqnarray*}
The remaining cases are:
\begin{itemize}
\item $\f$ satisfies (RG3) but not (LG3) so that $c = \beta f_L \bar{\delta} \bar{\gamma} (\bar{\tau}_R \sigma_R) \alpha$, where since $\tau_R$ is an inverse homotopy letter then $(\bar{\tau}_R \sigma_R)$ is a direct homotopy letter;
\item $\f$ satisfies (LG3) but not (RG3) so that  $c = \beta (\sigma_L \bar{\tau}_L) \bar{\delta} \bar{\gamma} \bar{f}_R \alpha$, where since $\sigma_L$ is an inverse homotopy letter then $(\sigma_L \bar{\tau}_L)$ is also an inverse homotopy letter;
\item $\f$ satisfies (RG3) and (LG3) so that $c = \beta (\sigma_L \bar{\tau}_L) \bar{\delta} \bar{\gamma} (\bar{\tau}_R \sigma_R) \alpha$, with the same observations as above.
\end{itemize}
In each case, we obtain $\partial (c) = m'+n'-l-l'+1 = \iota (c)$. Hence $c$ is a homotopy band, as claimed.

The remainder of the proof is the same as the proof of Theorem~\ref{thm:graph-maps} except that we have to take into account the scalar $\pm \lambda \mu^{-1}$. First let us deal with the sign. Consider a homotopy band $\sigma = \sigma_m \cdots \sigma_1$ and place a minus sign on any two homotopy letters, say $\sigma_j$ and $\sigma_i$ with $m \geq j > i \geq 1$. Then the following diagram induces an isomorphism of band complexes.
\[
\xymatrix@!R=4pt{
\cdots \arr^-{\sigma_1} & \xydot \arr^-{\sigma_m} \ar@{=}[d] & \xydot \ar@{.}[r] \ar@{=}[d] & \xydot \arr^-{-\sigma_j} \ar@{=}[d] & \xydot \arr^-{\sigma_{j-1}} \ar[d]^-{-1} & \xydot \ar@{.}[r] \ar[d]^-{-1} & \xydot \arr^-{\sigma_{i+1}} \ar[d]^-{-1} & \xydot \arr^-{-\sigma_i} \ar[d]^-{-1} & \xydot \ar@{.}[r] \ar@{=}[d] & \xydot \arr^-{\sigma_1} \ar@{=}[d] & \xydot \ar@{=}[d] \arr^-{\sigma_m} & \cdots \\
\cdots \arr_-{\sigma_1} & \xydot \arr_-{\sigma_m}            & \xydot \ar@{.}[r]            & \xydot \arr_-{\sigma_j}             & \xydot \arr_-{\sigma_{j-1}}             & \xydot \ar@{.}[r]              & \xydot \arr_-{\sigma_{i+1}}             & \xydot \arr_-{\sigma_i}               & \xydot \ar@{.}[r]            & \xydot \arr_-{\sigma_1}            & \xydot \arr_-{\sigma_m} & \cdots
}
\]
Now the unfolded diagram of the complex homotopic to $\M_{\f}$ given by the arguments of Theorem~\ref{thm:graph-maps} using Lemma~\ref{lem:homotopy} and Corollary~\ref{cor:homotopy} corresponds precisely to the homotopy band $c$. In this unfolded diagram, the homotopy letters carrying a minus sign are precisely those of the homotopy substrings $\alpha$ and $\beta$. The number of homotopy letters in $\alpha$ and $\beta$ is equal to $m - k+2$, which is even if and only if $k = \length{\rho}$ is even. In the case that $\length{\rho}$ is even, all the minus signs can be eliminated by an isomorphism as above. In the case that $\length{\rho}$ is odd, all but one minus sign can be removed by such an isomorphism. This deals with the signs.

To deal with the scalar $\lambda \mu^{-1}$ note that a scalar can be moved along to an adjacent homotopy letter as indicated in the following diagram.
\[
\xymatrix@!R=4pt{
\ar@{~}[r] & \xydot \ar@{=}[d] \ar[r]^-{\lambda \sigma_i} & \xydot \ar[d]^-{\lambda^{-1}} \ar[r]^-{\sigma_{i-1}}  & \xydot \ar@{=}[d] \ar@{~}[r] & \\
\ar@{~}[r] & \xydot \ar[r]_-{\sigma_i}                    & \xydot \ar[r]_-{\lambda \sigma_{i-1}}                & \xydot \ar@{~}[r]           &
}
\quad
\xymatrix@!R=4pt{
\ar@{~}[r] & \xydot \ar@{=}[d] \ar[r]^-{\lambda \sigma_i} & \xydot \ar[d]^-{\lambda^{-1}} \ar@{<-}[r]^-{\sigma_{i-1}}  & \xydot \ar@{=}[d] \ar@{~}[r] & \\
\ar@{~}[r] & \xydot \ar[r]_-{\sigma_i}                    & \xydot \ar@{<-}[r]_-{\lambda^{-1} \sigma_{i-1}}            & \xydot \ar@{~}[r]           &
} 
\]
This completes the proof.  
\end{proof}

It is useful to illustrate the proof of Proposition~\ref{prop:graph-band-to-band} in an example.

\begin{example} \label{instructive-example}
Let $\Lambda$ be the algebra given by the following quiver with relations.
\[
\begin{tikzpicture}
\node (A1) at (0,0){$1$};
\node (A2) at (0,1){$2$};
\node (A3) at (0,2){$3$};
\node (B2) at (1,1){$5$};
\node (B3) at (1,2){$4$};
\node (C1) at (2,0){$8$};
\node (C2) at (2,1){$7$};
\node (C3) at (2,2){$6$};
\path[->,font=\scriptsize,>=angle 90]
(A1)edge node[left]{$a$}(A2)
(A2)edge node[left]{$b$}(A3)
(B2)edge node[right]{$d$}(B3)
(C1)edge node[right]{$i$}(C2)
(C2)edge node[right]{$g$}(C3)
(A3)edge node[above]{$c$}(B3)
(B3)edge node[above]{$f$}(C3)
(A2)edge node[below]{$e$}(B2)
(B2)edge node[below]{$h$}(C2)
(A1)edge node[above]{$j$}(C1);
\draw[thick,dotted] (1.3,2) arc (0:170:.3cm) 
(2,.7) arc (270:450:.3cm) 
(1,1.3) arc (90:180:.3cm)
(2,.3) arc (90:180:.3cm)
(0,.7) arc (270:360:.3cm)
(0,1.7) arc (270:360:.3cm);
\end{tikzpicture}
\] 
Consider the following homotopy bands: $\sigma = \bar{e} \bar{d} c b$, $\tau = \bar{j} \bar{i} \bar{g} f c ba$. We consider graph maps $\f \colon \B_{\sigma,\lambda} \to \B_{\tau, \mu}$ defined by the following unfolded diagram.
\[
\xymatrix@!R=4px{
\B_{\sigma,\lambda} \colon \ar[d]_-{\f} &            &                          &\xydot               & \ar[l]_-{\lambda^{-1} \bar{e}} \xydot                              & \ar[l]_-{\bar{d}} \xydot \ar@{=}[d] \ar[r]^-{c} & \xydot \ar@{=}[d] \ar[r]^-{b} & \xydot \ar[d]^-{a} & \ar[l]_-{\lambda^{-1} \bar{e}}  \xydot \\
\B_{\tau, \mu} \colon                              & \xydot & \ar[l]^-{\bar{j}} \xydot & \ar[l]^-{\bar{i}} \xydot & \ar[l]^-{\mu^{-1} \bar{g}} \xydot  \ar[r]_-{f} & \xydot \ar[r]_-{c}                                 & \xydot \ar[r]_-{ba}                             & \xydot                                      & \ar[l]^-{\bar{j}} \xydot
}
\]
Let $c_\tau = \overline{e} \overline{fd} g i j \overline{a}$.
Let us work through the calculation of $\M_{\f}$ using the techniques of Lemma~\ref{lem:homotopy} and Corollary~\ref{cor:homotopy}.
\[
\xymatrix@!R=4px{
  &                          &\xydot               & \ar[l]_-{-\lambda^{-1} \bar{e}} \xydot                              & \ar[l]_-{-\bar{d}} \xydot \ar@{=}[dr] \ar[r]^-{-c} & \xydot \ar@{=}[dr] \ar[r]^-{b} & \xydot \ar[dr]^-{a} & \ar[l]_-{-\lambda^{-1} \bar{e}}  \xydot & \\
 & \xydot & \ar[l]^-{\bar{j}} \xydot & \ar[l]^-{\bar{i}} \xydot & \ar[l]^-{\mu^{-1} \bar{g}} \xydot  \ar[r]_-{f} & \xydot \ar[r]_-{c}                                 & \xydot \ar[r]_-{ba}                             & \xydot                                      & \ar[l]^-{\bar{j}} \xydot
}
\]
Removing the isomorphisms using Corollary~\ref{cor:homotopy} we get the following unfolded diagram of a homotopic complex, where $\circ$ denotes the removed endpoints of the isomorphisms.
\[
\xymatrix@!R=4px{
  &                          &\xydot               & \ar[l]_-{-\lambda^{-1} \bar{e}} \xydot                              & \circ & \circ & \xydot \ar[dr]^-{a} & \ar[l]_-{-\lambda^{-1} \bar{e}}  \xydot & \\
 & \xydot & \ar[l]^-{\bar{j}} \xydot & \ar[l]^-{\bar{i}} \xydot & \ar[l]^-{\mu^{-1} \bar{g}} \xydot  \ar[ul]_-{\bar{d} \bar{f}} & \circ                                 & \circ                             & \xydot                                      & \ar[l]^-{\bar{j}} \xydot
}
\]
Straightening this out, we get the following unfolded diagram,
\[
\xymatrix{
\cdots & \ar[l]_-{-\lambda^{-1} \bar{e}} \xydot & \ar[l]_-{\bar{d} \bar{f}} \xydot \ar[r]^-{\mu^{-1} g} & \xydot \ar[r]^-{i} & \xydot \ar[r]^-{j} & \xydot & \ar[l]_-{\bar{a}} & \ar[l]_-{-\lambda^{-1} \bar{e}} \xydot & \ar[l]_-{\bar{d} \bar{f}} \cdots,
}
\]
which gives rise to a band complex isomorphic to a band complex with the unfolded diagram,
\[
\xymatrix{
\cdots & \ar[l]_-{\bar{e}} \xydot & \ar[l]_-{\bar{d} \bar{f}} \xydot \ar[r]^-{-\lambda \mu^{-1} g} & \xydot \ar[r]^-{i} & \xydot \ar[r]^-{j} & \xydot & \ar[l]_-{\bar{a}} & \ar[l]_-{\bar{e}} \xydot & \ar[l]_-{\bar{d} \bar{f}} \cdots.
}
\]
\end{example}

We now state the corresponding statements when only one of $\sigma$ or $\tau$ is a homotopy band. In these cases $\M_{\f}$ is a string complex and there are no signs or scalars to take care of.

\begin{proposition} 
Let $\sigma = \beta \sigma_L \rho \sigma_R \alpha$ be a homotopy band and $\tau = \delta \tau_L \rho \tau_R \gamma$ be a homotopy string. Suppose $\f \colon \B_{\sigma, \lambda} \to \P_\tau$ is a compatibly oriented graph map. Then $\M_{\f}$ is isomorphic to the string complex $\P_c$, where
$c = \delta \tau_L \bar{\sigma}_L \bar{\beta} \bar{\alpha} \bar{\sigma}_R \tau_R \gamma$\footnote{The extended version of this statement covering the case when $\rho$ is longer than the homotopy band can be found in \cite[Thm. 2.2]{Addendum}}.
\end{proposition}

\begin{proposition} 
Let $\sigma = \beta \sigma_L \rho \sigma_R \alpha$ be a homotopy string and $\tau = \delta \tau_L \rho \tau_R \gamma$ be a homotopy band. Suppose $\f \colon \P_\sigma \to \B_\tau$ is a compatibly oriented graph map. Then $\M_{\f}$ is isomorphic to the string complex $\P_c$, where
$c = \beta \sigma_L \bar{\tau}_L \bar{\delta} \bar{\gamma} \bar{\tau}_R \sigma_R \alpha$\footnote{The extended version of this statement covering the case when $\rho$ is longer than the homotopy band can be found in \cite[Thm. 2.4]{Addendum}}.
\end{proposition}

\section{Mapping cones of single maps} \label{sec:single-maps}

In this section, we describe the mapping cone calculus for single maps. As in Section~\ref{sec:graph-maps}, we first describe the unfolded diagram of the mapping cone and then describe the mapping cone calculus in the case of a single map between string complexes. We then deal with the cases in which (at least) one of the indecomposable complexes is a band complex.

\subsection{The unfolded diagram of the mapping cone of a single map}

Let $\f \colon \Q_\sigma \to \Q_\tau$ be a single map with single component $f$. We illustrate the unfolded diagram of $\M_{\f}$ in the generic situation.
\[
\xymatrix@!R=4px{
\ar@{~}[r]^-{-\beta}   & \xydot \arr^-{-\sigma_L} & \xydot \ar[dr]^{f} \ar[r]^-{-\sigma_R} & \xydot  \ar@{~}[r]^-{-\alpha}                                                                      & & \\
& \ar@{~}[r]_-{\delta}  & \xydot \arr_-{\tau_L}     & \xydot                                                        & \ar[l]^-{\tau_R} \xydot \ar@{~}[r]_-{\gamma} &
                }
\]

\subsection{Mapping cones of single maps between string complexes}

We start with a technical definition analogous to Definition~\ref{def:graph-compatible}.

\begin{definition} \label{def:single-compatible}
Let $\sigma$ and $\tau$ be homotopy strings and suppose $\f \colon \P_\sigma \to \P_\tau$ is a (not necessarily singleton) single map sitting in the following orientation.
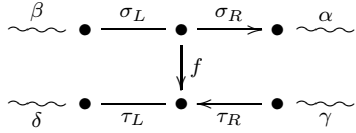
\[
\xymatrix@!R=4px{
\ar@{~}[r]^-{\beta}  & \xydot \arr^-{\sigma_L} & \xydot \ar[r]^-{\sigma_R} \ar[d]^-{f} & \xydot \ar@{~}[r]^-{\alpha}                        & \\
\ar@{~}[r]_-{\delta} & \xydot \arr_-{\tau_L}      & \xydot                                             & \ar[l]^-{\tau_R} \xydot \ar@{~}[r]_-{\gamma} & 
}
\]
We shall say that the homotopy strings $\sigma$ and $\tau$ are \emph{compatibly oriented (for $\f$)} if 
\begin{enumerate}[label=(\roman*)]
\item the homotopy letters $\sigma_L$ and $\tau_L$ do not contain $f$ as a subletter; 
\item if $\sigma_L$ is direct then $\sigma_L f = 0$ and if $\tau_L$ is inverse $f \bar{\tau}_L = 0$;
\item $\sigma_R = f f_R$ and $\tau_R = \bar{f} \bar{f}_L$ for some (possibly trivial) paths $f_R$ and $f_L$ in $(Q,I)$.
\end{enumerate}
\end{definition}

We note that the presentation of a singleton single map given in Section~\ref{sec:single-def} is already compatibly oriented. For a general single map we simply allow $f_L$ or $f_R$ to be trivial.

\begin{theorem} \label{thm:single-maps}
Let $\f \colon \P_\sigma \to \P_\tau$ be a singleton single map with single component $f$. Suppose that $\sigma = \beta \sigma_L \sigma_R \alpha$ and $\tau = \delta \tau_L \tau_R \gamma$ are compatibly oriented for $f$. Then $\M_f \simeq \P_{c_1} \oplus \P_{c_2}$, where
\[
c_1 = \begin{cases}
\beta \sigma_L f \overline{\tau_L} \overline{\delta} & \text{if $\sigma_L \neq \emptyset$ and $\tau_L \neq \emptyset$;} \\
\beta \sigma_L f                                     & \text{if $\sigma_L \neq \emptyset$ and $\tau_L = \emptyset$;} \\
f \overline{\tau_L} \overline{\delta}                & \text{if $\sigma_L = \emptyset$ and $\tau_L \neq \emptyset$;} \\
f                                                    & \text{if $\sigma_L = \emptyset$ and $\tau_L = \emptyset$,}
\end{cases}
\quad \text{and} \quad
c_2 = \begin{cases}
\overline{\gamma} f_L f f_R \alpha & \text{if $\sigma_R \neq \emptyset$ and $\tau_R \neq \emptyset$;} \\
\alpha                             & \text{if $\sigma_R \neq \emptyset$ and $\tau_R = \emptyset$;} \\
\overline{\gamma}                  & \text{if $\sigma_R  = \emptyset$ and $\tau_R \neq \emptyset$;} \\
\emptyset                          & \text{if $\sigma_R = \emptyset$ and $\tau_R = \emptyset$.}
\end{cases}
\]
\end{theorem}

\begin{remark}
Observe that the form of $c_2$ in the case that $\sigma_R = \emptyset$ or $\tau_R = \emptyset$ is not what one would expect from naive word combinatorics.
\end{remark}

\begin{proof}
The strategy of the proof is to define an isomorphism $\i \colon \P_{c_1} \oplus \P_{c_2} \to \M_{\f}$ directly. This isomorphism is easier to see at the level of unfolded diagrams. For word combinatorial reasons it is useful to state Theorem~\ref{thm:single-maps} using the compatible orientation of Definition~\ref{def:single-compatible}. However, the unfolded diagrams of the mapping cones look `more like complexes' and are thus easier to work with if the opposite orientation is taken for $\tau$. For clarity, the diagram below indicates how the map looks using this choice of orientation.

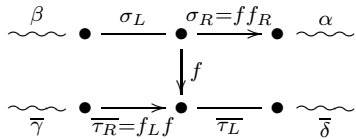
\[
\xymatrix@!R=4px{
\ar@{~}[r]^-{\beta}                       & \xydot \arr^-{\sigma_L}                           & \xydot \ar[d]^-{f} \ar[r]^-{\sigma_R = f f_R}  & \xydot \ar@{~}[r]^-{\alpha}                &   \\
\ar@{~}[r]_-{\overline{\gamma}}  & \xydot \ar[r]_-{\overline{\tau_R} = f_L f}  & \xydot \arr_-{\overline{\tau_L}}                   & \xydot \ar@{~}[r]_-{\overline{\delta}} & 
}
\]
For convenience, assume $\beta = \sigma_{m} \cdots \sigma_{i+2}$, $\sigma_L = \sigma_{i+1}$, $\sigma_R = \sigma_i$, 
$\alpha = \sigma_{i-1} \cdots \sigma_1$  and $\gamma = \tau_1 \cdots \tau_{j-1}$,  $\tau_R = \tau_j$, $\tau_L = \tau_{j+1}$, $\delta = \tau_n  \cdots \tau_{j+2}$. Then the unfolded diagram of the mapping cone $\M_{\f}$ is 

 
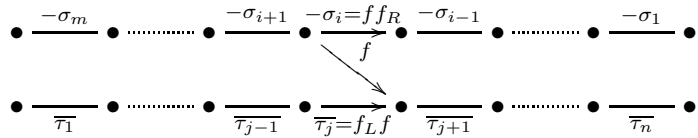
\[
\xymatrix@!R=4px{
\xydot \arr^-{-\sigma_m}           &  \xydot \ar@{.}[r]  & \xydot \arr^-{-\sigma_{i+1}}          & \xydot \ar[dr]^-{f} \ar[r]^-{-\sigma_i = f f_R}  & \xydot \arr^-{-\sigma_{i-1}}            & \xydot \ar@{.}[r]  & \xydot \arr^-{-\sigma_1}                & \xydot  \\
\xydot \arr_-{\overline{\tau_1}}  & \xydot \ar@{.}[r]  & \xydot \arr_-{\overline{\tau_{j-1}}} & \xydot \ar[r]_-{\overline{\tau_j} = f_L f}        & \xydot \arr_-{\overline{\tau_{j+1}}} & \xydot \ar@{.}[r]  & \xydot  \arr_-{\overline{\tau_n}} & \xydot 
}
\]
We define the maps $\i_1 \colon \P_{c_1} \to \M_{\f}$ and $\i_2 \colon \P_{c_2} \to \M_{\f}$, respectively, in the figures below.
In the figure, we write $x_i = e(\sigma_i)$, $x_o = s(\sigma_1)$, $y_j = e(\tau_j)$ and $y_0 = s(\tau_1)$.
\begin{figure}[H]
\begin{tikzcd}[row sep=scriptsize,column sep=scriptsize,nodes={scale=.72}]
P(x_m) \arrow[r, dash,"\sigma_{m}"]
\arrow[dd,color=blue,"\mp1"]&
P(x_{m\!-\!1})\arrow[r, dash,densely dotted]
\arrow[dd,color=blue,"\pm1"]&
P(x_{i\!+\!1})\arrow[r, dash,"\sigma_{i\!+\!1}"]
\arrow[dd,color=blue,"-1"]&
P(x_{i})\arrow[rrd, "f"]
\arrow[dd,color=blue,"+1"]&
&&&&
\\
&&&&&
P(y_j) \arrow[r, dash,"\overline\tau_{j\!+\!1}"]
\arrow[dd,color=blue,crossing over,pos=.8,bend left=45,swap,"+1"]
\arrow[d,color=red,swap,"-f_R"]&
P(y_{j\!+\!1})\arrow[r, dash,densely dotted]
\arrow[dd,color=blue,crossing over,pos=.8,bend left=45,swap,"+1"]&
P(y_{n\!-\!1})\arrow[r, dash,"\overline\tau_{n}"]
\arrow[dd,color=blue,crossing over,pos=.8,bend left=45,swap,"+1"]&
P(y_n)
\arrow[dd,color=blue,crossing over,pos=.8,bend left=45,swap,"+1"]&
\\
P(x_m) \arrow[r, dash,"-\sigma_{m}"]&
P(x_{m\!-\!1})\arrow[r, dash,densely dotted]&
P(x_{i\!+\!1})\arrow[r, dash,"-\sigma_{i\!+\!1}"]&
P(x_{i})\arrow[rr,"-\sigma_{i}=-ff_R"]
\arrow[rrd, "f"]&&
P(x_{i\!-\!1})\arrow[r, dash,"-\sigma_{i\!-\!1}"]&
P(x_{i\!-\!2})\arrow[r, dash,densely dotted]&
P(x_1)\arrow[r, dash,"-\sigma_{1}"]&
P(x_0)
\\
P(y_0)\arrow[r, dash,swap,"\overline\tau_{1}"]&
P(y_1)\arrow[r, dash,densely dotted]&
P(y_{j\!-\!2})\arrow[r, dash,swap,"\overline\tau_{j\!-\!1}"]&
P(y_{j\!-\!1})\arrow[rr, swap,"\overline\tau_{j}=f_Lf"]&&
P(y_j)\arrow[r, dash,swap,"\overline\tau_{j\!+\!1}"]&
P(y_{j\!+\!1})\arrow[r, dash,densely dotted]&
P(y_{n\!-\!1})\arrow[r, dash,swap,"\overline\tau_{n}"]&
P(y_n)
\end{tikzcd}
\caption{Unfolded diagram for $\i_1 \colon \P_{c_1} \to \M_{\f}$.} \label{c_1}
\end{figure}
\begin{figure}[H]
\begin{tikzcd}[row sep=scriptsize,column sep=scriptsize,nodes={scale=.72}]
&&&&&
P(x_{i\!-\!1}) \arrow[r, dash,"\sigma_{i\!-\!1}"]
\arrow[dd,color=blue,"+1"]&
P(x_{i\!-\!2})\arrow[r, dash,densely dotted]
\arrow[dd,color=blue,"-1"]&
P(x_{1})\arrow[r, dash,"\sigma_{1}"]
\arrow[dd,color=blue,"\mp1"]&
P(x_{0})
\arrow[dd,color=blue,"\pm1"]
\\
P(y_0) \arrow[r, dash,"\overline\tau_{1}"]
\arrow[dd,color=blue,crossing over,pos=.8,bend right=45,"+1"]&
P(y_{1})\arrow[r, dash,densely dotted]
\arrow[dd,color=blue,crossing over,pos=.8,bend right=45,"+1"]&
P(y_{j\!-\!2})\arrow[r, dash,"\overline\tau_{j\!-\!1}"]
\arrow[dd,color=blue,crossing over,pos=.8,bend right=45,"+1"]&
P(y_{j\!-\!1})
\arrow[rru,"f_Lff_R"]
\arrow[d,color=red,"-f_L"]
\arrow[dd,color=blue,crossing over,pos=.8,bend right=45,"+1"]
&&&&&&
\\
P(x_m) \arrow[r, dash,"-\sigma_{m}"]&
P(x_{m\!-\!1})\arrow[r, dash,densely dotted]&
P(x_{i\!+\!1})\arrow[r, dash,"-\sigma_{i\!+\!1}"]&
P(x_{i})\arrow[rr,"-\sigma_{i}=-ff_R"]
\arrow[rrd, "f"]&&
P(x_{i\!-\!1})\arrow[r, dash,"-\sigma_{i\!-\!1}"]&
P(x_{i\!-\!2})\arrow[r, dash,densely dotted]&
P(x_1)\arrow[r, dash,"-\sigma_{1}"]&
P(x_0)
\\
P(y_0)\arrow[r, dash,swap,"\overline\tau_{1}"]&
P(y_1)\arrow[r, dash,densely dotted]&
P(y_{j\!-\!2})\arrow[r, dash,swap,"\overline\tau_{j\!-\!1}"]&
P(y_{j\!-\!1})\arrow[rr, swap,"\overline\tau_{j}=f_Lf"]&&
P(y_j)\arrow[r, dash,swap,"\overline\tau_{j\!+\!1}"]&
P(y_{j\!+\!1})\arrow[r, dash,densely dotted]&
P(y_{n\!-\!1})\arrow[r, dash,swap,"\overline\tau_{n}"]&
P(y_n)
\end{tikzcd}
\caption{Unfolded diagram for $\i_2 \colon \P_{c_2} \to \M_{\f}$.} \label{c_2}
\end{figure}

\noindent
It is straightforward to check that these induce well-defined morphisms of complexes: only the commutativity of the central part of the diagram in each case is not immediately clear. 
The central part of the diagram defining $\i_1 \colon \P_{c_1} \to \M_{\f}$ is,
\[
\xymatrix{
P(x_i) \ar[rr]^-{f} \ar[d]_-{\begin{pmat} 1 & 0 \end{pmat}}                       &  & P(y_j) \ar[d]^-{\begin{pmat} -f_R & 1 \end{pmat}} \\
P(x_i) \oplus P(y_{j-1}) \ar[rr]_-{\begin{pmat} -ff_R & f \\ 0 & f_L f \end{pmat}} &  & P(x_{i-1}) \oplus P(y_j),
}
\]
which is clearly commutative. We need to check that $f_R \sigma_{i-1} = 0$ in the case that $\sigma_{i-1}$ is direct. But this is clear since in that case $\sigma_i \sigma_{i-1} = 0$ and $\sigma_i = f f_R$ forcing $f_R \sigma_{i-1} = 0$.
For the central part of the diagram defining $\i_2 \colon \P_{c_2} \to \M_{\f}$ we have,
\[
\xymatrix{
P(y_{j-1}) \ar[rr]^-{f_L f f_R} \ar[d]_-{\begin{pmat} -f_L & 1 \end{pmat}}         & & P(x_{i-1}) \ar[d]^-{\begin{pmat} 1 & 0 \end{pmat}} \\
P(x_i) \oplus P(y_{j-1}) \ar[rr]_-{\begin{pmat} -ff_R & f \\ 0 & f_L f \end{pmat}} & & P(x_{i-1}) \oplus P(y_j),
}
\]
which is clearly commutative. We need to check that $f_L \overline{\sigma_{i+1}} = 0$ in the case that $\sigma_{i+1} = \sigma_L$ is inverse. Again, this is clear since $\sigma_i = \sigma_R = ff_R$ together with $f_L f \neq 0$ and gentleness imply that $f_L \overline{\sigma_{i+1}} = 0$. 
Putting together the two diagrams, we get the following,
\[
\xymatrix{
P(x_i) \oplus P(y_{j-1}) \ar[rr]^-{\begin{pmat} f & 0 \\ 0 & f_L f f_R \end{pmat}} \ar[d]_-{\begin{pmat} 1 & 0 \\ -f_L & 1 \end{pmat}}  &  & P(y_j) \oplus P(x_{i-1}) \ar[d]^-{\begin{pmat} -f_R & 1 \\ 1 & 0 \end{pmat}} \\
P(x_i) \oplus P(y_{j-1}) \ar[rr]_-{\begin{pmat} -ff_R & f \\ 0 & f_L f \end{pmat}}                                                      &  & P(x_{i-1}) \oplus P(y_j),
}
\]
where the downward maps are clearly full rank matrices. In particular,
$\i = \begin{pmat} \i_1 \\ \i_2 \end{pmat} \colon \P_{c_1} \oplus \P_{c_2} \to \M_{\f}$ consists of full rank matrices in each degree, and is therefore an isomorphism of complexes.

We note that in the case that $\sigma_R = \emptyset$, the right hand side of the second diagram above does not exist and $\overline{\tau_1} \cdots \overline{\tau_{j-1}} = \overline{\gamma}$ and $c_2 = \gamma$. Similarly, if $\tau_R = \emptyset$ then the left hand side of the same diagram does not exist and $\alpha = \sigma_{i-1} \cdots \sigma_1$ and $c_2 = \alpha$. If both $\sigma_R = \emptyset$ and $\tau_R = \emptyset$ then the top line of this diagram does not exist and $c_2 = \emptyset$.
\end{proof}

\subsection{Mapping cones of single maps involving a band complex} \label{sec:single-band}

Suppose $\sigma$ and $\tau$ are homotopy strings or bands, with at least one being a homotopy band. We now consider the mapping cones of single maps $\f \colon \Q_\sigma \to \Q_\tau$. As was the case in Section~\ref{sec:graph-band}, $\M_{\f}$ is now indecomposable. Moreover, $\M_{\f}$ is a band complex if and only if $\sigma$ and $\tau$ are homotopy bands. This is the situation with which we start.

If both $\sigma$ and $\tau$ are homotopy bands, then any single map $\f \colon \B_{\sigma, \lambda} \to \B_{\tau, \mu}$ is necessarily of type (iv). We again impose the convention that the scalars $\lambda$ and $\mu$ are placed on direct arrows of $\sigma$ and $\tau$, respectively.

\begin{proposition} \label{prop:single-band-to-band}
Suppose $\sigma$ and $\tau$ are homotopy bands and $\f \colon \B_{\sigma,\lambda} \to \B_{\tau, \mu}$ with single component $f$. Suppose that $\sigma = \beta \sigma_L \sigma_R \alpha$ and $\tau = \delta \tau_L \tau_R \gamma$ are compatibly oriented for $\f$. Then $\M_{\f}$ is isomorphic to a band complex $\B_{c, -\lambda \mu^{-1}}$, where $c = \beta \sigma_L f \overline{\tau_L} \overline{\delta} \overline{\gamma} \overline{\tau_R} \overline{f} \sigma_R \alpha$ and where the scalar $-\lambda \mu^{-1}$ is placed on a direct homotopy letter of $\alpha$, $\beta$, $\overline{\gamma}$ or $\overline{\delta}$\footnote{It is possible here that the word $c$ is a nontrivial power of band. The statement when this occurs is given in \cite[Prop. 4.1]{Addendum}}.
\end{proposition}

Note that in the statement above $\overline{\tau_R} \overline{f} \sigma_R = f_L f \overline{f} f f_R = f_L f f_R$.

\begin{proof}
The verification that $c$ is indeed a homotopy band is similar to that in the proof of Proposition~\ref{prop:graph-band-to-band}. The construction of an isomorphism $\i \colon \B_{c,-\lambda \mu^{-1}} \to \M_{\f}$ proceeds exactly as in the proof of Theorem~\ref{thm:single-maps}.
The sign occurs by observing that one of the homotopy substrings $\sigma_m \cdots \sigma_{i+1}$ and $\sigma_{i-1} \cdots \sigma_1$ has an even number of homotopy letters and the other an odd number of homotopy letters since $\sigma$ has in total an even number of homotopy letters and only $\sigma_i$ has been removed. Therefore the signs on the identity morphisms with domain $P(x_m) = P(x_0)$ in the unfolded diagrams in Figures~\ref{c_1} and \ref{c_2} are different. However, $\i_1$ and $\i_2$ are `glued together' to form $\i$ by identifying these two maps with opposite signs. This is only possible if a minus sign is introduced to an odd number of homotopy letters of $\beta \sigma_L$ and $\alpha$ to make it that these maps have the same sign. By the argument of Theorem~\ref{thm:single-maps} this odd number of minus signs can then be reduced to one, which we position on the direct homotopy letter bearing the scalar.
\end{proof}

The following two propositions deal with the cases where precisely one of $\sigma$ or $\tau$ is a homotopy band and the other is a homotopy string. In these cases $\M_{\f}$ is a string complex. The proofs proceed exactly as in Theorem~\ref{thm:single-maps}. Note that when $\sigma$ is a homotopy band we have $\sigma_L \neq \emptyset$ and $\sigma_R \neq \emptyset$.

\begin{proposition} \label{prop:single-band-to-string}
Suppose $\sigma$ is a homotopy band and $\tau$ is a homotopy string. Suppose $\f \colon \B_{\sigma,\lambda} \to \P_\tau$ is a single map with single component $f$. Suppose that $\sigma = \beta \sigma_L \sigma_R \alpha$ and $\tau = \delta \tau_L \tau_R \gamma$ is compatibly oriented for $\f$. Then $\M_{\f} \simeq \P_c$, where
\[
c = \begin{cases}
\alpha \beta \sigma_L f \overline{\tau_L} \overline{\delta}                                                           & \text{if $\tau_R = \emptyset$;} \\
\overline{\gamma} \overline{\tau_R} \overline{f} \sigma_R \alpha \beta \sigma_L f \overline{\tau_L} \overline{\delta} & \text{if $\tau_R \neq \emptyset$.}
\end{cases}
\]
\end{proposition}

\begin{proposition} \label{prop:single-string-to-band}
Suppose $\sigma$ is a homotopy string and $\tau$ is a homotopy band. Suppose $\f \colon \P_\sigma \to \B_{\tau,\mu}$ is a single map with single component $f$. Suppose that $\sigma = \beta \sigma_L \sigma_R \alpha$ and $\tau = \delta \tau_L \tau_R \gamma$ is compatibly oriented for $\f$. Then $\M_{\f} \simeq \P_c$, where
\[
c = \begin{cases}
\beta \sigma_L f \overline{\tau_L} \overline{\delta} \overline{\gamma}                                                & \text{if $\sigma_R = \emptyset$;} \\
 \beta \sigma_L f \overline{\tau_L} \overline{\delta} \overline{\gamma} \overline{\tau_R} \overline{f} \sigma_R \alpha & \text{if $\sigma_R \neq \emptyset$.}
\end{cases}
\]
\end{proposition}

\section{Mapping cones of double maps} \label{sec:double-maps}

We now turn to the statement for double maps. Note that double maps are automatically `compatibly oriented' and therefore we do not require such a definition in this case.

\begin{theorem}\label{thm:double-maps}
Let $\f \colon \P_\sigma \longrightarrow \P_\tau$ be a double map with components $(f_L,f_R)$. Suppose that $\sigma = \beta \sigma_L \sigma_C \sigma_R \alpha$ and $\tau = \delta \tau_L \tau_C \tau_R \gamma$. Then $\M_{\f} \simeq \P_{c_1} \oplus \P_{c_2}$, where
$c_1 = \overline{\gamma} \overline{\tau_R} \overline{f_R} \sigma_R \alpha$
and
$c_2 = \beta \sigma_L f_L \overline{\tau_L} \overline{\delta}$.
\end{theorem}

\begin{proof}
In terms of unfolded diagrams, the map $\f$ is of the following form, where $f$ is permitted to be a trivial path. The case when $f$ is not trivial corresponds to a singleton double map.
\[
\xymatrix@!R=4px{
 \ar@{~}[r]^-{\beta}  & \xydot \arr^-{\sigma_L} & \xydot \ar[r]^-{\sigma_C = f_Lf} \ar[d]_-{f_L} & \xydot \arr^-{\sigma_R} \ar[d]^-{f_R} & \xydot \ar@{~}[r]^-{\alpha}    &   \\
 \ar@{~}[r]_-{\delta} & \xydot \arr_-{\tau_L}     & \xydot \ar[r]_-{\tau_C = ff_R}                         & \xydot \arr_-{\tau_R}                         & \xydot \ar@{~}[r]_-{\gamma} & 
 }
\]
For convenience, assume $\sigma_C = \sigma_i$ and $\tau_C = \tau_j$ etc. Then 
the unfolded diagram of the mapping cone $\M_{\f}$ is  
\[
\xymatrix@!R=4px{
 \xydot   \arr^-{-\sigma_{m}} & \xydot \ar@{.}[r] & \xydot \arr^-{-\sigma_{i+1}} & \xydot \ar[r]^-{-\sigma_i=-f_L f} \ar[dr]_-{f_L} & \xydot \arr^-{\sigma_{i-1}} \ar[dr]^-{f_R} & \xydot \ar@{.}[r] &  \xydot \arr^-{\sigma_1} &  \xydot  \\
  & \xydot \arr_-{\tau_{n}}     & \xydot  \ar@{.}[r] & \xydot \arr_-{\tau_{j+1}}     & \xydot \ar[r]_-{\tau_j=ff_R}                         & \xydot \arr_-{\tau_{j-1}}                          & \xydot \ar@{.}[r] & \xydot \arr_-{\tau_1} &  \xydot 
}
\]

The proof is the same as that of Theorem~\ref{thm:single-maps}, therefore, we just write down the maps $\i_1 \colon \P_{c_1} \to \M_{\f}$ and $\i_2 \colon \P_{c_2} \to \M_{\f}$ at the level of unfolded diagrams in the figures below. It is then straightforward to check these define full rank matrices in each cohomological degree, so that $\i = \begin{pmat} \i_1 \\ \i_2 \end{pmat} \colon \P_{c_1} \oplus \P_{c_2} \to \M_{\f}$ is an isomorphism.

\begin{figure}[H]
\begin{tikzcd}[column sep=normal, row sep=small,nodes={scale=.8}]
&
P(x_{i\!-\!1})\arrow[r,dash,"\sigma_{i\!-\!1}"]
\arrow[rd,"f_R"]
\arrow[ddd,color=blue,swap,bend right=40,pos=.6,"+1"]&
P(x_{i\!-\!2})\arrow[r,dash,"\sigma_{i\!-\!2}"]
\arrow[ddd,color=blue,swap,bend right=55,pos=.6,"-1"]&
P(x_{i\!-\!3})\arrow[r,dash,densely dotted]
\arrow[ddd,color=blue,swap,bend right=55,pos=.3,"+1"]&
P(x_{2})\arrow[r,dash,"\sigma_{2}"]
\arrow[ddd,color=blue,swap,bend right=55,pos=.3,"\pm1"]&
P(x_{1})\arrow[r,"\sigma_1"]
\arrow[ddd,color=blue,swap,bend right=55,pos=.3,"\mp1"]
&
P(x_0)\arrow[ddd,color=blue,swap,bend right=55,pos=.3,"\pm1"]
\\
&
&
P(y_{j\!-\!1})
\arrow[r,dash,swap,"\tau_{j\!-\!1}"]
\arrow[ddd,color=blue,bend left=55,pos=.3,swap,"+1"]&
P(y_{j\!-\!2})\arrow[r,dash,densely dotted]
\arrow[ddd,color=blue,bend left=55,pos=.3,swap,"+1"]&
P(y_{2})\arrow[r,dash,swap,"\tau_{2}"]
\arrow[ddd,color=blue,bend left=55,pos=.3,swap,"+1"]&
P(y_{1})\arrow[r,dash,swap,"\tau_{1}"]
\arrow[ddd,color=blue,bend left=55,pos=.3,swap,"+1"]&
P(y_{0})\arrow[ddd,color=blue,bend left=55,pos=.3,swap,"+1"]
\\
&&&&&&\\
P(x_i)\arrow[r,"-\sigma_i=-f_Lf"]\arrow[rd,swap,"f_L"]&
P(x_{i\!-\!1})\arrow[r,dash,"-\sigma_{i\!-\!1}"]\arrow[rd,"f_R"]&
P(x_{i\!-\!2})\arrow[r,dash,"-\sigma_{i\!-\!2}"]&
P(x_{i\!-\!3})\arrow[r,dash,densely dotted]&
P(x_{2})\arrow[r,dash,"-\sigma_{2}"]&
P(x_{1})\arrow[r,dash,"-\sigma_1"]&
P(x_0)
\\
&
P(y_j)\arrow[r,swap,"ff_R"]&
P(y_{j\!-\!1})\arrow[r,dash,swap,"\tau_{j\!-\!1}"]&
P(y_{j\!-\!2})\arrow[r,dash,densely dotted]&
P(y_{2})\arrow[r,dash,swap,"\tau_{2}"]&
P(y_{1})\arrow[r,dash,swap,"\tau_{1}"]&
P(y_{0})
\end{tikzcd}
\end{figure}

\begin{figure}[H]
\begin{tikzcd}[column sep=normal, row sep=small,nodes={scale=.8}]
P(x_m)\arrow[r,dash,"\sigma_{m}"]
\arrow[ddd,color=blue,bend right=55,pos=.3,"\mp1"]&
P(x_{m\!-\!1})\arrow[r,dash,densely dotted]
\arrow[ddd,color=blue,swap,bend right=55,pos=.3,"\pm1"]&
P(x_{i\!+\!2})\arrow[r,dash,"\sigma_{i\!+\!2}"]
\arrow[ddd,color=blue,swap,bend right=55,pos=.3,"+1"]&
P(x_{i\!+\!1})\arrow[r,dash,"\sigma_{i\!+\!1}"]
\arrow[ddd,color=blue,swap,bend right=55,pos=.3,"-1"]&
P(x_{i})\arrow[rd,"f_L"]
\arrow[ddd,color=blue,swap,bend right=55,pos=.3,"+1"]&
&\\
&
P(y_n)\arrow[r,dash,swap,"\tau_{n}"]
\arrow[ddd,color=blue,bend left=55,pos=.3,swap,"+1"]&
P(y_{n\!-\!1})\arrow[r,dash,densely dotted]
\arrow[ddd,color=blue,bend left=55,pos=.3,swap,"+1"]&
P(y_{j\!+\!2})\arrow[r,dash,swap,"\tau_{j\!+\!2}"]
\arrow[ddd,color=blue,bend left=55,pos=.3,swap,"+1"]&
P(y_{j\!+\!1})\arrow[r,dash,swap,"\tau_{j\!+\!1}"]
\arrow[ddd,color=blue,bend left=55,pos=.3,swap,"+1"]&
P(y_{j})\arrow[dd,color=red,"-f"]
\arrow[ddd,color=blue,bend left=55,pos=.3,"+1"]&
\\
&&&&&&\\
P(x_m)\arrow[r,dash,"-\sigma_{m}"]&P(x_{m\!-\!1})\arrow[r,dash,densely dotted]&P(x_{i\!+\!2})\arrow[r,dash,"-\sigma_{i\!+\!2}"]&P(x_{i\!+\!1})\arrow[r,dash,"-\sigma_{i\!+\!1}"]&
P(x_{i})\arrow[r,"-f_Lf"]\arrow[rd,"f_L"]&
P(x_{i\!-\!1})\arrow[rd,"f_R"]&\\
&P(y_n)\arrow[r,dash,swap,"\tau_{n}"]&P(y_{n\!-\!1})\arrow[r,dash,densely dotted]&P(y_{j\!+\!2})\arrow[r,dash,swap,"\tau_{j\!+\!2}"]&P(y_{j\!+\!1})\arrow[r,dash,swap,"\tau_{j\!+\!1}"]&
P(y_{j})\arrow[r,swap,"ff_R"]&
P(y_{j\!-\!1})
\end{tikzcd}
\end{figure}
We note that the signs on the isomorphisms $P(x_0) \to P(x_0)$ and $P(x_m) \to P(x_m)$ defined in the above diagram are necessarily different because one of $\sigma_m \cdots \sigma_{i+1}$ and $\sigma_{i-1} \cdots \sigma_1$ has an even number of homotopy letters and the other an odd number of homotopy letters.
\end{proof}

We now give the analogous statements for double maps $\f \colon \Q_\sigma \to \Q_\tau$ where at least one of $\sigma$ or $\tau$ is a homotopy band. As was the case for graph maps and single maps before, $\M_{\f}$ is indecomposable and is a band complex if and only if both $\sigma$ and $\tau$ are homotopy bands. 
In this case we get an extra sign on the scalar because of the different parities in the lengths of $\sigma_m \cdots \sigma_{i+1}$ and $\sigma_{i-1} \cdots \sigma_1$ noted above.

\begin{proposition} \label{prop:double-band-to-band}
Let $\sigma = \beta \sigma_L \sigma_C \sigma_R \alpha$ and $\tau = \delta \tau_L \tau_C \tau_R \gamma$ be homotopy bands. Suppose $\f \colon \B_{\sigma,\lambda} \to \B_{\tau,\mu}$ is a double map with components $(f_L,f_R)$. Then $\M_{\f} \simeq \B_{c,-\lambda \mu^{-1}}$, where $c = \beta \sigma_L f_L \bar{\tau}_L \bar{\delta} \bar{\gamma} \bar{\tau}_R \bar{f}_R \sigma_R \alpha$ and where the scalar $-\lambda \mu^{-1}$ is placed on a direct homotopy letter of $\alpha$, $\beta$, $\overline{\gamma}$ or $\overline{\delta}$\footnote{It is possible here that the word $c$ is a nontrivial power of band. The statement when this occurs is given in \cite[Prop. 4.2]{Addendum}}.
\end{proposition}

\begin{proposition} \label{prop:double-band-to-string}
Let $\sigma = \beta \sigma_L \sigma_C \sigma_R \alpha$ be a homotopy band and $\tau = \delta \tau_L \tau_C \tau_R \gamma$ be a homotopy string. Suppose $\f \colon \B_{\sigma, \lambda} \to \P_\tau$ is a double map with components $(f_L,f_R)$. Then $\M_{\f} \simeq \P_c$, where
$c = \delta \tau_L f_L \bar{\sigma}_L \bar{\beta} \bar{\alpha} \bar{\sigma}_R \bar{f}_R \tau_R \gamma$.
\end{proposition}

\begin{proposition} \label{prop:double-string-to-band}
Let $\sigma = \beta \sigma_L \sigma_C \sigma_R \alpha$ be a homotopy string and $\tau = \delta \tau_L \tau_C \tau_R \gamma$ be a homotopy band. Suppose $\f \colon \P_\sigma \to \B_{\tau, \mu}$ is a double map with components $(f_L,f_R)$. Then $\M_{\f} \simeq \P_c$, where
$c = \beta \sigma_L f_L \bar{\tau}_L \bar{\delta} \bar{\gamma} \bar{\tau}_R \bar{f}_R \sigma_R \alpha$.
\end{proposition}

\section{Mapping cones and quasi-graph maps} \label{sec:quasi-graph-maps}

Single maps and double maps that are not singleton occur in a homotopy class that is determined by a quasi-graph map $\P_\sigma \rightsquigarrow \Sigma^{-1} \P_\tau$. In this case, it is possible to read off the mapping cone of any representative of the homotopy class from the quasi-graph map. Before we state how this is done, we need the following bookkeeping definition to cover the case when a quasi-graph map is supported in precisely one cohomological degree.

\begin{definition} \label{def:qgm-compatible}
Let $\sigma$ and $\tau$ be homotopy strings or bands and suppose $\varphi \colon \Q_\sigma \rightsquigarrow \Sigma^{-1} \Q_\tau$ is a quasi-graph map supported in exactly one degree, i.e. corresponds to the following diagram
\[
\xymatrix@!R=4px{
\ar@{~}[r]^-{\beta}  & \xydot                         & \ar@{-}[l]_-{\sigma_L} \xydot \ar@{=}[d] \ar@{-}[r]^-{\sigma_R} & \xydot \ar@{~}[r]^-{\alpha}                             &  \\
\ar@{~}[r]_-{\delta} & \xydot \ar@{-}[r]_-{\tau_L} & \xydot                                                                           & \ar@{-}[l]^-{\tau_R} \xydot \ar@{~}[r]_-{\gamma} & 
}
\]
We say that the homotopy strings or bands $\sigma$ and $\tau$ are \emph{compatibly oriented on the left} if
\begin{itemize}
\item[(1)] If $\sigma_L$ is direct then $\tau_L$ is either empty, inverse with $\sigma_L \tau_L \neq 0$ or direct with $\tau_L = \sigma'_L \sigma_L$ for some nontrivial $\sigma'_L$. 
\item[(2)] If $\sigma_L$ is inverse then $\tau_L$ is inverse and $\sigma_L = \tau_L \tau'_L$ for some nontrivial $\tau'_L$. 
\item[(3)] If $\sigma_L = \emptyset$ then $\tau_L$ is inverse. 
\end{itemize}
Similarly, we say that the homotopy strings or bands $\sigma$ and $\tau$ are \emph{compatibly oriented on the right} if the following dual conditions hold: 
\begin{itemize}
\item[(1)] If $\sigma_R$ is inverse then $\tau_R$ is either empty, direct with $\sigma_R \tau_R \neq 0$ or inverse with $\tau_R = \sigma'_R \sigma_R$ for some nontrivial $\sigma'_R$. 
\item[(2)] If $\sigma_R$ is direct then $\tau_R$ is direct and $\sigma_R = \tau_R \tau'_R$ for some nontrivial $\tau'_R$. 
\item[(3)] If $\sigma_R =\emptyset$ then $\tau_R$ is direct. 
\end{itemize}
We say that the homotopy strings or bands $\sigma$ and $\tau$ are \emph{compatibly oriented for $\varphi$} if they are compatibly oriented on the left and on the right. 
\end{definition}

Note that when the maximal common homotopy substring $\rho$ determining a quasi-graph map $\varphi \colon \Q_\sigma \rightsquigarrow \Sigma^{-1} \Q_\tau$ is of length at least one,  the homotopy strings or bands $\sigma$ and $\tau$ are automatically compatibly oriented for $\varphi$ in an unfolded diagram of $\varphi$.

\begin{proposition}\label{prop:quasi-graph-map}
Let $\sigma$ and $\tau$ be homotopy strings or bands. Suppose $\varphi \colon \Q_\sigma \rightsquigarrow \Sigma^{-1} \Q_\tau$ is a quasi-graph map determined by a maximal common homotopy substring $\rho$, i.e. $\sigma = \beta \sigma_L \rho \sigma_R \alpha$ and $\tau = \delta \tau_L \rho \tau_R \gamma$. Assume further that $\sigma$ and $\tau$ are compatibly oriented for $\varphi$. Suppose $\f \colon \Q_\sigma \to \Q_\tau$ is a representative of the homotopy set determined by $\varphi$\footnote{The extension of statements $(3)$ and $(4)$ to the case that $\rho$ is longer than the homotopy band is given in \cite[Prop. 3.1]{Addendum}. The extension of statement $(2)$ to the case $\rho$ is longer than at least one of the homotopy bands together the discussion of the case that $c$ is a nontrivial power of a homotopy band is given in \cite[Thm. 3.2]{Addendum}}. 
\begin{enumerate}
\item If $\sigma$ and $\tau$ are homotopy strings then $\M_{\f}$ is isomorphic to $\P_{c_1} \oplus \P_{c_2}$, where $c_1 = \beta \sigma_L \rho \tau_R \gamma$ and $c_2 = \delta \tau_L \rho \sigma_R \alpha$.
\item If $(\sigma, \lambda)$ and $(\tau, \mu)$ are homotopy bands then $\M_{\f}$ is isomorphic to $\B_{c, -\lambda \mu^{-1}}$ , where $c = \beta \sigma_L \rho \tau_R \gamma \delta \tau_L \rho \sigma_R \alpha$ and where the scalar $-\lambda \mu^{-1}$ is placed on a direct homotopy letter of $\alpha$ or $\beta$ or an inverse homotopy letter of $\gamma$ or $\delta$.
\item If $(\sigma, \lambda)$ is a homotopy band and $\tau$ is a homotopy string then $\M_{\f}$ is isomorphic to $\P_{c}$, where $c = \delta \tau_L \rho \sigma_R \alpha \beta \sigma_L \rho \tau_R \gamma$.
\item If $\sigma$ is a homotopy string and $(\tau,\mu)$ is a homotopy band then $\M_{\f}$ is isomorphic to $\P_{c}$, where $c = \beta \sigma_L \rho \tau_R \gamma \delta \tau_L \rho \sigma_R \alpha$.
\end{enumerate}
\end{proposition}

\begin{proof}
For ease of notation, we assume $\sigma$ and $\tau$ are homotopy strings. Let $\varphi \colon \P_\sigma \rightsquigarrow \Sigma^{-1} \P_\tau$ be a quasi-graph map determined by a maximal common homotopy string $\rho$. 
We note that by orienting the unfolded diagrams so that they are compatibly oriented for $\varphi$ we have interchanged the roles of $\gamma$ and $\delta$ and $\tau_L$ and $\tau_R$ vis-\`a-vis the statements in Sections~\ref{sec:single-maps} and \ref{sec:double-maps}.
Firstly, if the homotopy class determined by $\varphi$ contains a double map, then we simply compute the mapping cone of such a double map as in Theorem~\ref{thm:double-maps} and Propositions~\ref{prop:double-band-to-band}, \ref{prop:double-band-to-string} and \ref{prop:double-string-to-band}. So we may assume that the homotopy class of $\varphi$ does not contain a double map. By \cite[Prop. 4.8]{ALP} this occurs if and only if quasi-graph map endpoint conditions (RQ3) and (LQ3) do not occur.

We first assume that $\rho$ is a trivial homotopy string and choose a single map representative of the homotopy class determined by $\varphi$. Since $\rho$ is trivial, this corresponds to the situation in Theorem~\ref{thm:single-maps} where precisely one of $f_L$ or $f_R$ is trivial. Assume $f_R$ is nontrivial and $f_L$ is trivial.
The left-hand diagram below indicates this situation, where the notation has been updated to reflect the set up of the quasi-graph map in Definition~\ref{def:qgm-compatible}, indicated on the right-hand side.
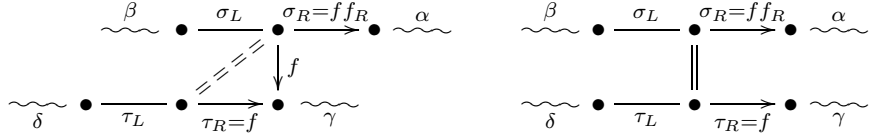
\[
\xymatrix@!R=4px{
                     & \ar@{~}[r]^-{\beta}    & \xydot \arr^-{\sigma_L}       & \xydot \ar[d]^-{f} \ar[r]^-{\sigma_R = f f_R} \ar@{==}[dl] & \xydot \ar@{~}[r]^-{\alpha}                &   \\
\ar@{~}[r]_-{\delta} & \xydot \arr_-{\tau_L}  & \xydot \ar[r]_-{\tau_R = f}  &  \xydot \ar@{~}[r]_-{\gamma}                               &                                            & 
}
\quad
\xymatrix@!R=4px{
\ar@{~}[r]^-{\beta}  & \xydot \arr^-{\sigma_L}     & \xydot \ar@{=}[d] \ar[r]^-{\sigma_R = f f_R} & \xydot \ar@{~}[r]^-{\alpha}  &  \\
\ar@{~}[r]_-{\delta} & \xydot \ar@{-}[r]_-{\tau_L} & \xydot \ar[r]_-{\tau_R = f}                  & \xydot \ar@{~}[r]_-{\gamma} & 
}
\]
Here $\sigma_L$ corresponds to the homotopy letter $\sigma_{i+1}$ and $\tau_L$ corresponds to $\overline{\tau_{j-1}}$ in the proof of Theorem~\ref{thm:single-maps}. Since $f_R$ is nontrivial, the map $\i_1 \colon \P_{c_1} \to \M_{\f}$ is defined as in Figure~\ref{c_1} in the proof of Theorem~\ref{thm:single-maps}.
However, the map $\i_2 \colon \P_{c_2} \to \M_{\f}$ defined in Figure~\ref{c_2} is only well defined in the case that $\sigma_L$ is direct. In the case that $\sigma_L$ is inverse, we need to modify the definition of $\i_2 \colon \P_{c_2} \to \M_{\f}$ as follows. First note that this means that $\varphi$ satisfies the left endpoint condition (LQ1), meaning that $\sigma_L = \overline{g'}\overline{g}$ and $\tau_L = \overline{g}$ for some nontrivial paths $g$ and $g'$ with $gg' \neq 0$. Then $\i_2 \colon \P_{c_2} \to \M_{\f}$ can be modified as in the diagram below, with the remaining components defined as in Figure~\ref{c_2}.
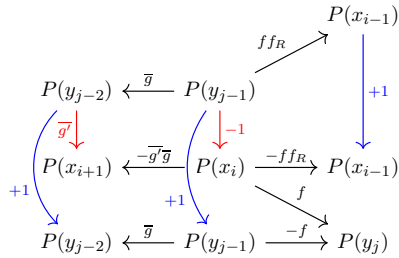
\begin{figure}[H]
\begin{tikzcd}[row sep=scriptsize,column sep=scriptsize,nodes={scale=.72}]
&&P(x_{i-1})\arrow[dd,color=blue,pos=.5,"+1"]
\\
P(y_{j-2}) \arrow[dd,color=blue,crossing over,pos=.6,bend right=45,swap,"+1"]  \arrow[d,color=red,pos=.5,swap,"\overline{g'}"] & \arrow[l,swap,"\overline{g}"] \arrow[dd,color=blue,crossing over,pos=.7,bend right=35,swap,"+1"]  P(y_{j-1}) \arrow[d,color=red,pos=.5,"-1"]\arrow[ru,"ff_R"]&
\\
P(x_{i+1}) & \arrow[l,swap,"-\overline{g'}\overline{g}"] P(x_i) \arrow[r,"-ff_R"] \arrow[rd,"f"] & P(x_{i-1})
\\
P(y_{j-2}) & \arrow[l,swap,"\overline{g}"] P(y_{j-1}) \arrow[r,"-f"]&P(y_j)
\end{tikzcd}
\end{figure}

\noindent
The argument when $f_R$ is trivial and $f_L$ is nontrivial is dual to the argument above.

Finally, in the case that both $f_L$ and $f_R$ are trivial, then $\length{\rho} > 0$. In this case, one modifies $\i_1$ above using identity morphisms with alternating signs along the top string until one reaches a direct homotopy letter or an endpoint as indicated in the second diagram above. Similarly for $\i_2$, where one proceeds until one reaches an inverse homotopy letter or an endpoint as the first diagram above. When one of $\sigma$ or $\tau$ is a homotopy band, one deals with the signs exactly as in Propositions~\ref{prop:single-band-to-band}, \ref{prop:single-band-to-string} and \ref{prop:single-string-to-band}.  
\end{proof}

\section{Examples}

In this section we will illustrate the graphical mapping cone calculus  in $\Db(\Lambda)$ developed in Sections~\ref{sec:graph-maps}-\ref{sec:quasi-graph-maps} on some concrete examples. In the first example, we consider maps involving only string complexes and in the second example we  consider band complexes.

\begin{example} Let $\Lambda$  be the gentle algebra given by the following quiver and relations: 
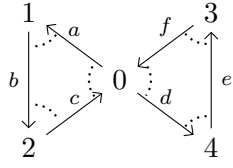
\[
\begin{tikzpicture}[scale=1.2]
\node (A) at (0,0) {$0$};
\node (B1) at (-1,.75) {$1$};
\node (B2) at (-1,-.75) {$2$};
\node (C1) at (1,.75) {$3$};
\node (C2) at (1,-.75) {$4$};
\path[->,font=\scriptsize,>=angle 90]
(A) edge node[above]{$a$} (B1)
(B1) edge node[left]{$b$} (B2)
(B2) edge node[above]{$c$} (A)
(C1) edge node[above]{$f$} (A)
(A) edge node[above]{$d$} (C2)
(C2) edge node[right]{$e$} (C1);
\draw[thick,dotted] (-.3,.15 ) arc (150:220:.3cm)
(.3,.15) arc (30:-35:.3cm)
(-.7,.5) arc (-35:-90:.3cm)
(-.7,-.4) arc (35:90:.3cm)
(.95,.4) arc (-90:-150:.3cm)
(.95,-.4) arc (90:150:.3cm);
\end{tikzpicture}
\]
 
(1) 
Consider the homotopy strings  $\sigma=edcba\bar d$ and $\tau=\bar e \bar f c b af e,$ and the graph map $\f \colon \P_\sigma \to \P_\tau$ given by
\[
\scalebox{.9}{\begin{tikzpicture}[scale=1.5]
\node (A1) at (0,0) {};
\node (A2) at (1.5,0) {$P(3)$};
\node (A3) at (3,0) {$P(4)$};
\node (A4) at (4.5,0) {$P(2)$};
\node (A5) at (6,0) {$P(1)$};
\node (A6) at (7.5,0) {$P(0)$};
\node (A7) at (9,0) {$P(4)$};
\node (A8) at (-1,0) {$P_\sigma$};
\node (B1) at (0,-1) {$P(4)$};
\node (B2) at (1.5,-1) {$P(3)$};
\node (B3) at (3,-1) {$P(0)$};
\node (B4) at (4.5,-1) {$P(2)$};
\node (B5) at (6,-1) {$P(1)$};
\node (B6) at (7.5,-1) {$P(3)$};
\node (B7) at (9,-1) {$P(4)$};
\node (B8) at (-1,-1) {$P_\tau$};
\path[->,font=\scriptsize,>=angle 90]
(A2) edge node[above]{$e$} (A3)
(A3) edge node[above]{$dc$} (A4)
(A4) edge node[above]{$b$} (A5)
(A5) edge node[above]{$a$} (A6)
(A7) edge node[above]{$d$} (A6)
(A8) edge (B8);
\path[->,font=\scriptsize,>=angle 90]
(A3) edge node[right]{$d$} (B3)
(A6) edge node[right]{$f$} (B6);
\path[font=\scriptsize,>=angle 90]
(A4)  edge[double]  (B4)
(A5)  edge[double]  (B5);
\path[->,font=\scriptsize,>=angle 90]
(B2) edge node[above]{$e$} (B1)
(B3) edge node[above]{$f$} (B2)
(B3) edge node[above]{$c$} (B4)
(B4) edge node[above]{$b$} (B5)
(B5) edge node[above]{$af$} (B6)
(B6) edge node[above]{$e$} (B7);
\end{tikzpicture}}
\]
By Theorem~\ref{thm:graph-maps},  the mapping cone $\M_{\f}$ is isomorphic to $\P_{c_1} \oplus \P_{c_2}$ where $c_1=d \bar a  a f e= dfe$  and $c_2  = ed c \bar c fe = edfe$  (cf. green and red boxes in the figure below).
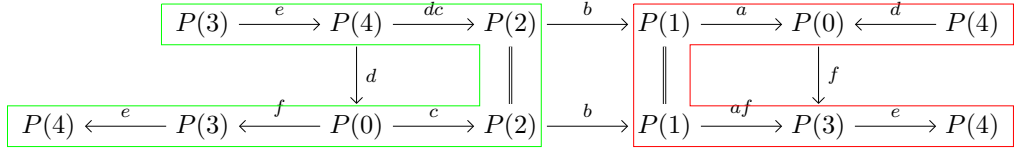
\[
\scalebox{.9}{\begin{tikzpicture}[scale=1.5]
\node (A1) at (0,0) {};
\node (A2) at (1.5,0) {$P(3)$};
\node (A3) at (3,0) {$P(4)$};
\node (A4) at (4.5,0) {$P(2)$};
\node (A5) at (6,0) {$P(1)$};
\node (A6) at (7.5,0) {$P(0)$};
\node (A7) at (9,0) {$P(4)$};
\node (B1) at (0,-1) {$P(4)$};
\node (B2) at (1.5,-1) {$P(3)$};
\node (B3) at (3,-1) {$P(0)$};
\node (B4) at (4.5,-1) {$P(2)$};
\node (B5) at (6,-1) {$P(1)$};
\node (B6) at (7.5,-1) {$P(3)$};
\node (B7) at (9,-1) {$P(4)$};
\path[->,font=\scriptsize ,>=angle 90]
(A2) edge node[above]{$e$} (A3)
(A3) edge node[above]{$dc$} (A4)
(A4) edge node[above]{$b$} (A5)
(A5) edge node[above]{$a$} (A6)
(A7) edge node[above]{$d$} (A6);
\path[->,font=\scriptsize ,>=angle 90]
(A3) edge node[right]{$d$} (B3)
(A6) edge node[right]{$f$} (B6);
\path[font=\scriptsize ,>=angle 90]
(A4)  edge[double]  (B4)
(A5)  edge[double]  (B5);
\path[->,font=\scriptsize ,>=angle 90]
(B2) edge node[above]{$e$} (B1)
(B3) edge node[above]{$f$} (B2)
(B3) edge node[above]{$c$} (B4)
(B4) edge node[above]{$b$} (B5)
(B5) edge node[above]{$af$} (B6)
(B6) edge node[above]{$e$} (B7);
\draw [color=green, dashed, line width=1.1] (1.1,.2)--(4.8,.2)--(4.8,-1.2)--(-.4, -1.2)--(-.4,-.8)--(4.2,-.8)--(4.2,-.2)--(1.1,-.2)--(1.1,.2);
\draw [color=red, dotted, line width=1.2] (5.7,.2)--(9.4,.2)--(9.4,-.2)--(6.25,-.2)--(6.25,-.8)--(9.4,-.8)--(9.4,-1.2)--(5.7,-1.2)--(5.7,.2);
\end{tikzpicture}}
\]

(2) 
Non-singleton double maps and single maps arise in the context of quasi-graph maps. As described in Section~\ref{sec:quasi-graph}, a given quasi-graph map gives rise to a class of (single and double) maps, which are all homotopy equivalent to each other. In particular, they all have the same mapping cone, which is the `mapping cone of the quasi-graph map'. We will now illustrate this with an example. Consider a quasi-graph map $\phi \colon \P_\sigma  \rightsquigarrow \Sigma^{-1} \P_\tau$, for homotopy strings  $\sigma=bacb$ and $\tau=\bar f cba$, given by
\[
\scalebox{.9}{\begin{tikzpicture}[scale=1.5]
\node (A1) at (0,0) {P(2)};
\node (A2) at (1.5,0) {$P(1)$};
\node (A3) at (3,0) {$P(0)$};
\node (A4) at (4.5,0) {$P(2)$};
\node (A5) at (6,0) {$P(1)$};
\node (A6) at (-1,0) {$P_\sigma$};
\node (B1) at (1.5,-1) {$P(3)$};
\node (B2) at (3,-1) {$P(0)$};
\node (B3) at (4.5,-1) {$P(2)$};
\node (B4) at (6,-1) {$P(1)$};
\node (B5) at (7.5,-1) {$P(0)$};
\node (B6) at (-1,-1) {$\Sigma^{-1} P_\tau$};
\path[->,font=\scriptsize,>=angle 90]
(A1) edge node[above]{$b$} (A2)
(A2) edge node[above]{$a$} (A3)
(A3) edge node[above]{$c$} (A4)
(A4) edge node[above]{$b$} (A5);
\draw [->,
line join=round,
decorate, decoration={
    zigzag,
    segment length=4,
    amplitude=.9,post=lineto,
    post length=2pt
}] (A6) -- (B6);
\path[->,font=\scriptsize,>=angle 90]
(B2) edge node[above]{$f$} (B1)
(B2) edge node[above]{$c$} (B3)
(B3) edge node[above]{$b$} (B4)
(B4) edge node[above]{$a$} (B5);
\path[font=\scriptsize,>=angle 90]
(A3)  edge[double]  (B2)
(A4)  edge[double]  (B3)
(A5)  edge[double]  (B4);
\end{tikzpicture}
}
\]
By Proposition~\ref{prop:quasi-graph-map}(1), the mapping cone of any single or double map $\f \colon \P_\sigma \to \P_\tau$ in the homotopy set determined by $\phi$ is isomorphic to $\P_{c_1} \oplus \P_{c_2}$, where $c_1=bacba$ and $c_2=\bar f cb$ (cf. green and red boxes in the figure below).
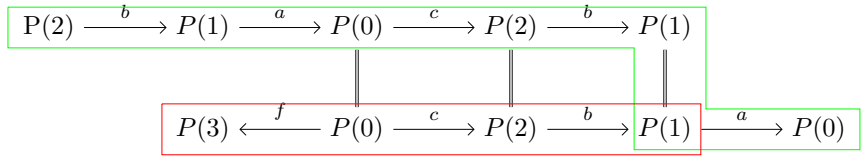
\[
\scalebox{.9}{\begin{tikzpicture}[scale=1.5]
\node (A1) at (0,0) {P(2)};
\node (A2) at (1.5,0) {$P(1)$};
\node (A3) at (3,0) {$P(0)$};
\node (A4) at (4.5,0) {$P(2)$};
\node (A5) at (6,0) {$P(1)$};
\node (B1) at (1.5,-1) {$P(3)$};
\node (B2) at (3,-1) {$P(0)$};
\node (B3) at (4.5,-1) {$P(2)$};
\node (B4) at (6,-1) {$P(1)$};
\node (B5) at (7.5,-1) {$P(0)$};
\path[->,font=\scriptsize ,>=angle 90]
(A1) edge node[above]{$b$} (A2)
(A2) edge node[above]{$a$} (A3)
(A3) edge node[above]{$c$} (A4)
(A4) edge node[above]{$b$} (A5);
\path[->,font=\scriptsize ,>=angle 90]
(B2) edge node[above]{$f$} (B1)
(B2) edge node[above]{$c$} (B3)
(B3) edge node[above]{$b$} (B4)
(B4) edge node[above]{$a$} (B5);
\path[font=\scriptsize ,>=angle 90]
(A3)  edge[double]  (B2)
(A4)  edge[double]  (B3)
(A5)  edge[double]  (B4);
\draw [color=green, dashed, line width=1.1] (-.4,.2)--(6.4,.2)--(6.4,-.8)--(7.9, -.8)--(7.9,-1.2)--(5.7,-1.2)--(5.7,-.2)--(-.4,-.2)--(-.4,.2);
\draw [color=red, dotted, line width=1.2] (1.1,-.75)--(6.35,-.75)--(6.35, -1.25)--(1.1,-1.25)--(1.1,-.75);
\end{tikzpicture}
}
\]
We now consider a single map $\f \colon \P_\sigma \to \P_\tau$ and a double map $\g \colon \P_\sigma \to \P_\tau$ in the homotopy set determined by the quasi-graph map $\phi$.
\begin{enumerate}[label=(\roman*)]
\item Let $\f: \P_\sigma \to \P_\tau$ be a single map in the homotopy set determined by $\phi$ given by
\[
\scalebox{.9}{\begin{tikzpicture}[scale=1.5]
\node (A1) at (0,0) {P(2)};
\node (A2) at (1.5,0) {$P(1)$};
\node (A3) at (3,0) {$P(0)$};
\node (A4) at (4.5,0) {$P(2)$};
\node (A5) at (6,0) {$P(1)$};
\node (A6) at (-1,0) {$P_\sigma$};
\node (B1) at (0,-1) {$P(0)$};
\node (B2) at (1.5,-1) {$P(1)$};
\node (B3) at (3,-1) {$P(2)$};
\node (B4) at (4.5,-1) {$P(0)$};
\node (B5) at (6,-1) {$P(3)$};
\node (B6) at (-1,-1) {$P_\tau$};
\path[->,font=\scriptsize,>=angle 90]
(A1) edge node[above]{$b$} (A2)
(A2) edge node[above]{$a$} (A3)
(A3) edge node[above]{$c$} (A4)
(A4) edge node[above]{$b$} (A5)
(A6) edge (B6);
\path[->,font=\scriptsize,>=angle 90]
(A3) edge node[right]{$c$} (B3);
\draw[thick,dotted] (2.7,0) arc (180:260:.35cm);
\draw[thick,dotted] (3,-.7) arc (90:175:.35cm);
\path[->,font=\scriptsize,>=angle 90]
(B2) edge node[above]{$a$} (B1)
(B3) edge node[above]{$b$} (B2)
(B4) edge node[above]{$c$} (B3)
(B4) edge node[above]{$f$} (B5);
\end{tikzpicture}
}
\]
where we have drawn the unfolded diagram so that it is compatibly oriented (see Definition~\ref{def:single-compatible}).
By Theorem~\ref{thm:single-maps}, the mapping cone is $\M_{\f}\cong\P_{c_1} \oplus \P_{c_2}$, where $c_1 = bacba$ and $c_2=\bar{f}cb$ (cf. green and red boxes in the figure below). 
\[
\scalebox{.9}{\begin{tikzpicture}[scale=1.5]
\node (A1) at (0,0) {$P(2)$};
\node (A2) at (1.5,0) {$P(1)$};
\node (A3) at (3,0) {$P(0)$};
\node (A4) at (4.5,0) {$P(2)$};
\node (A5) at (6,0) {$P(1)$};
\node (B1) at (0,-1) {$P(0)$};
\node (B2) at (1.5,-1) {$P(1)$};
\node (B3) at (3,-1) {$P(2)$};
\node (B4) at (4.5,-1) {$P(0)$};
\node (B5) at (6,-1) {$P(3)$};
\path[->,font=\scriptsize ,>=angle 90]
(A1) edge node[above]{$b$} (A2)
(A2) edge node[above]{$a$} (A3)
(A3) edge node[above]{$c$} (A4)
(A4) edge node[above]{$b$} (A5);
\path[->,font=\scriptsize ,>=angle 90]
(A3) edge node[right]{$c$} (B3);
\draw [thick,dotted] (2.7,0) arc (180:260:.35cm);
\draw [thick,dotted] (3,-.7) arc (90:175:.35cm);
\path[->,font=\scriptsize ,>=angle 90]
(B2) edge node[above]{$a$} (B1)
(B3) edge node[above]{$b$} (B2)
(B4) edge node[above]{$c$} (B3)
(B4) edge node[above]{$f$} (B5);
\draw [color=green, dashed, line width=1.2] (-.4,.2)--(3.3,.2)--(3.3,-1.2)--(-.4, -1.2)--(-.4,-.8)--(2.7,-.8)--(2.7,-.2)--(-.4,-.2)--(-.4,.2);
\draw [color=red, dotted, line width=1.2] (4.2,.2)--(6.4,.2)--(6.4,-.2)--(4.8, -.2)--(4.8,-.8)--(6.4,-.8)--(6.4,-1.2)--(4.2,-1.2)--(4.2,.2);
\path[->,font=\scriptsize ,>=angle 90] (B4) edge node[right]{$c$} (A4);
\end{tikzpicture}}
\]
\item Let $\g \colon \P_\sigma \to \P_\tau$ be a double map in the homotopy set determined by $\phi$ given by 
\[
\scalebox{.9}{\begin{tikzpicture}[scale=1.5]
\node (A1) at (0,0) {};
\node (A2) at (1.5,0) {};
\node (A3) at (3,0) {$P(2)$};
\node (A4) at (4.5,0) {$P(1)$};
\node (A5) at (6,0) {$P(0)$};
\node (A6) at (7.5,0) {$P(2)$};
\node (A7) at (9,0) {$P(1)$};
\node (A8) at (-1,0) {$P_\sigma$};
\node (B1) at (0,-1) {$P(0)$};
\node (B2) at (1.5,-1) {$P(1)$};
\node (B3) at (3,-1) {$P(2)$};
\node (B4) at (4.5,-1) {$P(0)$};
\node (B5) at (6,-1) {$P(3)$};
\node (B6) at (7.5,0) {};
\node (B7) at (9,0) {};
\node (B8) at (-1,-1) {$P_\tau$};
\path[->,font=\scriptsize,>=angle 90]
(A3) edge node[above]{$b$} (A4)
(A4) edge node[above]{$a$} (A5)
(A5) edge node[above]{$c$} (A6)
(A6) edge node[above]{$b$} (A7)
(A8) edge (B8);
\path[->,font=\scriptsize,>=angle 90]
(A4) edge node[right]{$a$} (B4)
(A5) edge node[right]{$f$} (B5);
\path[->,font=\scriptsize,>=angle 90]
(B2) edge node[above]{$a$} (B1)
(B3) edge node[above]{$b$} (B2)
(B4) edge node[above]{$c$} (B3)
(B4) edge node[above]{$f$} (B5);
\end{tikzpicture}}
\]
By Theorem~\ref{thm:double-maps}, its mapping cone is $\M_{\g}\cong\P_{c_1} \oplus \P_{c_2}$, where $c_1 =\bar f cb$ and $c_2=bacba$ (cf. red and green boxes in the figure below). 
\[
\scalebox{.9}{\begin{tikzpicture}[scale=1.4]
\node (A1) at (0,0) {};
\node (A2) at (1.5,0) {};
\node (A3) at (3,0) {$P(2)$};
\node (A4) at (4.5,0) {$P(1)$};
\node (A5) at (6,0) {$P(0)$};
\node (A6) at (7.5,0) {$P(2)$};
\node (A7) at (9,0) {$P(1)$};
\node (B1) at (0,-1) {$P(0)$};
\node (B2) at (1.5,-1) {$P(1)$};
\node (B3) at (3,-1) {$P(2)$};
\node (B4) at (4.5,-1) {$P(0)$};
\node (B5) at (6,-1) {$P(3)$};
\node (B6) at (7.5,0) {};
\node (B7) at (9,0) {};
\path[->,font=\scriptsize ,>=angle 90]
(A3) edge node[above]{$b$} (A4)
(A4) edge node[above]{$a$} (A5)
(A5) edge node[above]{$c$} (A6)
(A6) edge node[above]{$b$} (A7);
\path[->,font=\scriptsize ,>=angle 90]
(A4) edge node[right]{$a$} (B4)
(A5) edge node[right]{$f$} (B5);
\path[->,font=\scriptsize ,>=angle 90]
(B2) edge node[above]{$a$} (B1)
(B3) edge node[above]{$b$} (B2)
(B4) edge node[above]{$c$} (B3)
(B4) edge node[above]{$f$} (B5);
\draw [color=green, dashed, line width=1.1] (2.7,.3)--(4.8,.3)--(4.8,-1.3)--(-.4, -1.3)--(-.4,-.8)--(4.2,-.8)--(4.2,-.2)--(2.7,-.2)--(2.7,.3);
\draw [color=red, dotted, line width=1.2] (5.7,.3)--(5.7,-1.3)--(6.3,-1.3)--(6.3,-.2)--(9.4,-.2)--(9.4,.3)--(5.7,.3);
\end{tikzpicture}}
\]
\end{enumerate}
\end{example}

We finally give an example involving band complexes.

\begin{example}
Let $\Lambda$ be the algebra given by the quiver with relations in Example~\ref{instructive-example}.
Let $\sigma= gh \bar d \bar f$ and  $\tau=b\bar e \bar d c$ be homotopy bands with corresponding scalars $\lambda$ and $\mu$ respectively. By Proposition~\ref{prop:graph-band-to-band}, the mapping cone for the graph map $\f \colon \B_{\sigma,\lambda} \to \B_{\tau,\mu} $ given below is $\M_{\f}\cong\B_{c,\lambda\mu^{-1}}$, where $ c=  gh e \bar b\bar c   d \bar d \bar f = ghe\bar b\bar c \bar f$. Graphically this corresponds to the following diagrams of unfolded complexes. 
\[
\scalebox{.9}{\begin{tikzpicture}[scale=1.5]
\node (A1) at (.4,0) {};
\node (A2) at (1.5,0) {$P(6)$};
\node (A3) at (3,0) {$P(5)$};
\node (A4) at (4.5,0) {$P(6)$};
\node (A5) at (5.5,0) {};
\node (B0) at (-1,-1) {};
\node (B1) at (0,-1) {$P(3)$};
\node (B2) at (1.5,-1) {$P(2)$};
\node (B3) at (3,-1) {$P(5)$};
\node (B4) at (4.5,-1) {$P(4)$};
\node (B5) at (6,-1) {$P(3)$};
\node (B6) at (7,-1) {};
\path[->,font=\scriptsize ,>=angle 90]
(A2) edge node[above]{$\lambda gh$} (A3)
(A4) edge node[above]{$d f$} (A3);
\draw [densely dotted] (A1) --(A2) (A4) --(A5)
(B0) --(B1) (B5) --(B6);
\draw (A3) edge[double] (B3);
\path[->,font=\scriptsize ,>=angle 90]
(A4) edge node[right]{$f$} (B4);
\path[->,font=\scriptsize ,>=angle 90]
(B1) edge node[above]{$b$} (B2)
(B3) edge node[above]{$e$} (B2)
(B4) edge node[above]{$d$} (B3)
(B4) edge node[above]{$\mu c$} (B5);
\path[->,font=\scriptsize ,>=angle 90,color=red]
(A2) edge node[left]{$\lambda ghe$} (B2);
\draw [color=red, dotted, line width=1.2] (.3,.3)--(1.8,.3)--(1.8,-1.3)
(.3,-.2)--(1,-.2)--(1,-.75)--(-1,-.75)
(-1,-1.3)--(1.8,-1.3)
(7,-.75)--(4.8,-.75)--(4.8,-.2)--(5.5,-.2)
(5.5,.3)--(4.2,.3)--(4.2,-1.3)--(7,-1.3);
\end{tikzpicture}}
\]

\end{example}

\bigskip


\end{document}